\documentclass[]{alggeom}

\usepackage{amsmath}
\usepackage{amsthm}
\usepackage{amsfonts}
\usepackage{amssymb}
\usepackage{tikz}
\usepackage{chngcntr}
\usepackage[utf8]{inputenc}
\usepackage[T1]{fontenc}
\usepackage{lmodern} 
\usepackage{hyperref}

\usetikzlibrary{matrix,arrows}

\DeclareMathOperator{\Ker}{Ker}
\DeclareMathOperator{\nul}{nul}
\DeclareMathOperator{\rk}{rk}  
\DeclareMathOperator{\vol}{vol}
\DeclareMathOperator{\spn}{span}
\DeclareMathOperator{\Coh}{Coh}

\DeclareMathOperator{\Spec}{Spec}
\newcommand{\unit}{1\!\!1}
\newcommand{\Lat}{\normalfont\textbf{Lat}}

\newtheorem{theorem}{Theorem}[section]
\newtheorem{corollary}[theorem]{Corollary}
\newtheorem{lemma}[theorem]{Lemma}
\newtheorem{proposition}[theorem]{Proposition}

\theoremstyle{definition}
\newtheorem{definition}[theorem]{Definition}

\theoremstyle{remark}
\newtheorem{remark}[theorem]{Remark}

\counterwithin*{equation}{section}

\begin{document}
	
\title{Numerical Cohomology}
\author{Thomas McMurray Price}
\email{tom.price.math@gmail.com}
\address{Toronto, Canada}

\classification{14G40, 18G99}
\keywords{arithmetic cohomology, homological algebra, Arakelov bundle, lattice}

\begin{abstract}

We develop a numerical approach to cohomology. Essentially, vector spaces and linear maps are replaced by real numbers, which represent dimensions of vector spaces and ranks of linear maps. We use this to refine ideas of Van der Geer and Schoof about the cohomology of Arakelov bundles.

\end{abstract}

\maketitle

\section{Introduction}

	It is well known that, to an arbitrary Arakelov divisor $D$, it is natural to associate a corresponding lattice $L$. In \cite{VDGS}, Van der Geer and Schoof propose a definition of $h^0(D)$ in terms of the Gaussian sum on $L$. This is ``the arithmetic analogue of the dimension of the vector space $H^0(D)$ of sections of the line bundle associated to a divisor $D$ on an algebraic curve.'' They also define $h^1(D)$ as $h^0(K - D)$, with $K$ the canonical divisor. For a higher-rank Arakelov bundle $M$ (in other words, a metrized $O_F$-module, with $F$ a number field), the quantities $h^0(M)$ and $h^1(M)$ can be defined in an analogous way.
	
	The ideas proposed in \cite{VDGS} add interesting new connections to the already rich analogy between number fields and function fields of curves; for example, the Poisson summation formula yields an analogue of the Riemann Roch formula. However, there is a drawback: although $h^0$ and $h^1$ behave in some ways like dimensions of cohomology vector spaces, the definition given in \cite{VDGS} is ad-hoc, doesn't resemble other approaches to cohomology, and has little hope of generalizing to higher dimensions since it depends on properties unique to curves.
	
	This paper is motivated by the following question: can the cohomology of Arakelov bundles be unified with other approaches to cohomology? Since these quantities can be arbitrary non-negative real numbers, they can't literally be realized as dimensions of vector spaces, so this rules out any straightforward approach. We address this by developing an approach to cohomology that cuts out the middleman of vector spaces and deals directly with quantities that act like vector space dimensions and ranks of linear maps.
		
	For simplicity, in this paper, we'll restrict our attention to Arakelov bundles over $\mathbb{Z}$. In other words, lattices. However, one can see without much difficulty how these ideas apply just as well to Arakelov bundles over other number fields. One could also note that, according to the number field - function field analogy, the cohomology of an Arakelov bundle should just be the same as the cohomology of the underlying lattice, since the direct image functor is exact for a finite morphism of curves.
	
	In section 2, we introduce the concept of a \textbf{numerical exact sequence}, which is a sequence of real numbers that behaves like the dimensions of vector spaces lying in an exact sequence. This allows us to define a \textbf{numerical $\delta$-functor}, a numerical analogue of the $\delta$-functors introduced by Grothendieck in \cite{Grothendieck}. We introduce the Euler characteristic function associated with a numerical $\delta$-functor. We also prove a numerical analogue of the fact that effaceable $\delta$-functors are universal. This gives us a condition which, as we'll see later, uniquely determines both the quantities $h^k$ of a lattice and $h^k$ of a coherent sheaf on a projective scheme over a field (see \ref{def:coh_exists}, \ref{def:coh_exists2}, \ref{thm:coh_exists}, and \ref{thm:coh_exists2}).
	
	In section 3, we state a few basic category-theoretic definitions and theorems. This is just for convenience in later sections; everything in this section is unoriginal and fairly trivial.
	
	In section 4, we introduce the concept of a \textbf{rank}. This is a function on the morphisms of a category which generalizes the rank of a linear map. For our purposes, the role of a left-exact functor will be played by a \textbf{left-exact rank}, a certain kind of rank which acts like the composite of a left-exact functor with the rank of a linear map.
	
	In section 5, we show how different ideas in previous chapters relate to coherent sheaves on a projective scheme over a field and their cohomology. We show that the quantities $h^k$ in this setting are uniquely determined by a condition introduced in section 4. We'll show later that the quantities $h^k$ of a lattice are uniquely determined by the same condition (see \ref{def:coh_exists}, \ref{def:coh_exists2}, \ref{thm:coh_exists}, and \ref{thm:coh_exists2}).
	
	In section 6, we introduce a numerical analogue of the zig-zag lemma, in a general category-theoretic setting. The main new ideas introduced in the previous chapters are numerical exact sequences and ranks; the numerical zig-zag lemma is where we see interaction between these two ideas.
	
	In section 7, we state a few basic theorems and definitions pertaining to lattices. As in section 3, this is mostly just for convenience in later sections.
	
	In section 8, we define a certain function on the morphisms of the category of lattices (this category was introduced in section 7). We apply a recent result of Regev and Stephens-Davidowitz \cite{RSD} to show that this function is a left-exact rank. This will allow us to apply the numerical zig-zag lemma to lattices in section 10.
	
	In section 9, we introduce \textbf{tight lattices}, and prove a few properties about them that we'll need later. Tightness is a sort of positivity condition for lattices. It's stable under tensor products; we'll also see later that sufficiently tight lattices have small $h^1$. Tight lattices are important in section 10: we make a tight resolution of a lattice as a step in defining its cohomology.
	
	In section 10, we define the sequence of functions $h = (h^k)$ on the objects of $\Lat$. Our approach resembles the general tendency in cohomology, where the cohomology of an object is defined by taking a resolution, applying a functor (or, in our case, a rank) to the resolution, and taking the cohomology of the resulting chain complex. We use the numerical zig-zag lemma to establish basic properties of $h$. Also, we show that the sequence of functions $h$ on $\Lat$ is, in a precise sense, analogous to $h$ on the category of coherent sheaves on a projective scheme over a field, since they are both uniquely determined by the same conditions (see \ref{def:coh_exists}, \ref{def:coh_exists2}, \ref{thm:coh_exists}, and \ref{thm:coh_exists2}).
	
	In section 11, we work out explicitly what $h^k$ and $\chi$ of a lattice is. We find that our definitions are consistent with the ones given in \cite{VDGS}.
	
	\subsection*{Related Work}
	
	This isn't the first paper that attempts to refine the definitions of $h^0$ and $h^1$ given in \cite{VDGS}. In \cite{Borisov}, Borisov introduces the concept of a ``ghost space'', which is essentially a group with fuzzy addition, where the sum of two elements is a measure rather than a point. He defines the cohomology of an Arakelov divisor as a certain sequence of ghost spaces. Another approach is given in \cite{Weng1} and \cite{Weng2}. For $\mathbb{A}$ the ring of adeles of a number field $F$, and $g$ in the general linear group $GL_r(\mathbb{A})$, Weng defines $H^0(F, g)$ and $H^1(F, g)$ as certain locally compact abelian groups related to $\mathbb{A}$. He defines the numbers $h^0(F, g)$ and $h^1(F, g)$ by integrating certain functions over $H^0(F, g)$ and $H^1(F, g)$ respectively. These ideas are applied in \cite{Weng2} to moduli spaces of semi-stable lattices and non-abelian zeta functions for number fields.
		
	This paper has similar motivations to the ones given above, but our approach has very little in common with them other than that. One of the main differences is that the papers above don't consider any analogue of the long exact sequence in cohomology associated with a short exact sequence, or of the zig-zag lemma. We do, and use this to give a condition that uniquely determines both the quantities $h^k$ of a lattice and $h^k$ of a coherent sheaf on a projective scheme over a field (see \ref{def:coh_exists}, \ref{def:coh_exists2}, \ref{thm:coh_exists}, and \ref{thm:coh_exists2}). We also require a property of the Gaussian function which isn't considered by the papers above, and in fact has only recently been proven (see \cite{RSD} and section 8 of this paper). We use it to show that a certain function on the morphisms of the category of lattices behaves like ``rank of the induced map on global sections''; this allows us to apply the numerical zig-zag lemma in section 10.

\section{\texorpdfstring{Numerical Exact Sequences and $\delta$-functors}%
	{Numerical Exact Sequences and Delta-functors}}

\begin{definition} \label{def:numex}
	A sequence $e_0, e_1, e_2, \ldots$ of real numbers is called a \textbf{numerical exact sequence} if all $e_k$ are nonnegative and the alternating sums $\sum\limits_{j=0}^{k} (-1)^j e_{k-j}$ are nonnegative for all $k \geq 0$. As an example (in fact, the example that inspired this definition), if we have an exact sequence of vector spaces ${0 \rightarrow V_1 \rightarrow V_2 \rightarrow \ldots}$  then $\dim(V_1), \dim(V_2), \ldots$ must be a numerical exact sequence: clearly all elements of the sequence are nonnegative, and the alternating sums represent the dimensions of the kernels of the morphisms, and therefore must be nonnegative.
\end{definition}

Let $O$ be a class and $E$ a class of ordered triples of elements of $O$. Later, we'll take $O$ to be the objects of a category and $E$ to be the triples $(A, B, C)$ for each short exact sequence $0 \rightarrow A \rightarrow B \rightarrow C \rightarrow 0$ in $O$, but we won't need any of that structure for this section.

\begin{definition} \label{def:numdelta}
	Suppose $T$ is a sequence of functions $(T^0, T^1, T^2, \ldots)$ from $O$ to the  nonnegative real numbers. We say $T$ is a \textbf{numerical $\delta$-functor} if, for any triple $(A, B, C) \in E$, the following is a numerical exact sequence:
	$$T^0(A), T^0(B), T^0(C), T^1(A), T^1(B), T^1(C), T^2(A) \ldots$$
\end{definition}

\begin{definition} \label{def:chi}
	Suppose that $T = (T^0, T^1, T^2, \ldots)$ is a numerical $\delta$-functor on $O$. Suppose additionally that, for any object $A \in O$, we have $T^i(A) = 0$ for sufficiently large $i$ (we'll say $T$ is \textbf{eventually $0$} whenever this holds). Then we define the \textbf{Euler characteristic} $\chi_T$ on $O$ as $\chi_T(A) = \sum\limits_{i=0}^{\infty} (-1)^i T^i(A)$. We'll sometimes refer to $\chi_T$ simply as $\chi$ when $T$ is implied by context.
\end{definition}

\begin{theorem} \label{thm:chi_additive}
	Suppose, as above, $T$ is a numerical $\delta$-functor and is eventually $0$. If $(A, B, C) \in E$, then $\chi_T(B) = \chi_T(A) + \chi_T(C)$.
\end{theorem}
\begin{proof}
	Let $(e_k)_{k=0}^\infty = T^0(A), T^0(B), T^0(C), T^1(A), T^1(B), \ldots$. We then have, for sufficiently large $k$, that $\sum\limits_{j=0}^{k} (-1)^j e_{j} = \chi(A) - \chi(B) + \chi(C)$. So, for sufficiently large even $k$, the related sum $\sum\limits_{j=0}^{k} (-1)^j e_{k-j}$ is $\chi(A) - \chi(B) + \chi(C)$; for sufficiently large odd $k$ it is $-\chi(A) + \chi(B) - \chi(C)$. We then have, since $T$ is a numerical $\delta$-functor, $\chi(A) - \chi(B) + \chi(C) \geq 0$ and $-\chi(A) + \chi(B) - \chi(C) \geq 0$. It follows immediately that $\chi(A) - \chi(B) + \chi(C) = 0$, and so $\chi(B) = \chi(A) + \chi(C)$.
\end{proof}

\begin{definition} \label{def:eff}
	If $F$ is a function from $O$ to the nonnegative reals, we say $F$ is \textbf{effaceable} if for any $A \in O$ and any $\epsilon > 0$, there exists a triple $(A, B, C) \in E$ such that $F(B) \leq \epsilon$.
\end{definition}

\begin{definition} \label{def:eff_functor}
	We say a numerical $\delta$-functor $T$ is \textbf{effaceable} if $T^i$ is effaceable for all $i \geq 1$.
\end{definition}

The above definitions are numerical analogues of $\delta$-functors and effaceability, which were introduced by Grothendieck in \cite{Grothendieck}, and are also explained in \cite{Hartshorne}. It can be shown that, if $S$ and $T$ are effaceable $\delta$-functors with $S^0 \simeq T^0$, then $S \simeq T$. We have the following numerical analogue of this:

\begin{theorem} \label{thm:eff_equal}
	Suppose $S$ and $T$ are effaceable numerical $\delta$-functors, and $S^0 = T^0$. Then $S$ = $T$.
\end{theorem}
\begin{proof}
Suppose $S^i = T^i$ for all $i < n$. For any $A \in O$ and any $\epsilon > 0$, we can find a triple $(A, B, C)$ with $S^n(B) \leq \epsilon$. From the induction hypothesis and the fact that $S$ is a numerical $\delta$-functor, we have the following inequalities:
\begin{equation} \label{eqn:ineq1}
S^n(B) - S^n(A) + S^{n-1}(C) - S^{n-1}(B) + S^{n-1}(A) + \ldots \geq 0
\end{equation}
\begin{equation} \label{eqn:ineq2}
T^n(A) - S^{n-1}(C) + S^{n-1}(B) - S^{n-1}(A) + \ldots \geq 0
\end{equation}

From adding \eqref{eqn:ineq1} and \eqref{eqn:ineq2}, we get that $S^n(B) - S^n(A) + T^n(A) \geq 0$, so $T^n(A) + \epsilon \geq T^n(A) + S^n(B) \geq S^n(A)$. Since this holds for any $\epsilon > 0$, $T^n(A) \geq S^n(A)$. By a similar argument, $S^n(A) \geq T^n(A)$, and so $S^n(A) = T^n(A)$. Since this holds for any $A$ in $O$, $S^n = T^n$. Therefore, by induction, $S = T$.
\end{proof}

\begin{definition} \label{def:coh_exists} 
	Suppose $\ell$ is a real-valued function on $O$. Then we say \textbf{cohomology exists} for $(O, E, \ell)$ if there exists an effaceable numerical $\delta$-functor $T = (T^0, T^1, \ldots)$ with $T^0 = \ell$. We have from \ref{thm:eff_equal} that, if such a $T$ exists, it is uniquely determined.
\end{definition}

\section{Category-Theoretic Preliminaries}

\begin{definition} \label{def:zero_func}
	Suppose  $\mathfrak{C}$ is a category with a zero object. For any objects $A$ and $B$ of $\mathfrak{C}$, there is a unique morphism from $A$ to $B$ that factors through a zero object. We'll denote it as $0_{AB}$, or sometimes simply as $0$ when it's clear that we're referring to a morphism from $A$ to $B$.
\end{definition}

\begin{definition} \label{def:chain}
	Suppose  $\mathfrak{C}$ is a category with a zero object, and we have a sequence of objects and morphisms $0 \xrightarrow{a_0} A_1 \xrightarrow{a_1} A_2 \xrightarrow{a_2} \ldots$ in $\mathfrak{C}$. We say this sequence is a \textbf{chain complex} if $a_{i+1} \circ a_i = 0$ for all i. We say that it's \textbf{exact} if, for all $i \geq 1$, the morphism $a_i$ factors as $g \circ f$ for some pair of morphisms $f$ and $g$, with $f$ a cokernel of $a_{i - 1}$ and $g$ a kernel of $a_{i+1}$. Clearly an exact sequence is also a chain complex. A \textbf{short exact sequence} is an exact sequence of the form $0 \rightarrow A \xrightarrow{f} B \xrightarrow{g} C \rightarrow 0$. In this case, $f$ must be a kernel of $g$ and $g$ must be a cokernel of $f$.
\end{definition}

It will be convenient to state here some basic facts about kernels that will be used later on. The next three results are fairly trivial and well-known so the proofs will be skipped.

\begin{proposition} \label{thm:zero_kernel}
	For any objects $A$ and $B$ of $\mathfrak{C}$, $\text{id}_A$ is a kernel of $0_{AB}$.
\end{proposition}

\begin{proposition} \label{thm:kernel_monic}
	Kernels are monic.
\end{proposition}

\begin{proposition} \label{thm:kernel_compose}
	If we have morphisms $f:A \rightarrow B$ and $g: B \rightarrow C$, and $g$ is monic, then a kernel of $f$ is also a kernel of $g \circ f$ and vice versa.
\end{proposition}

\begin{definition} \label{def:cat_kernels}
	We say that a category $\mathfrak{C}$ is a \textbf{category with kernels} if it has a zero object and every morphism has a kernel. For a morphism $f$ in $\mathfrak{C}$, we'll use the notation $\ker(f)$ to denote a kernel morphism of $f$, and $\Ker(f)$ to denote a kernel object, i.e. the domain of a kernel morphism. For the rest of this section, all objects and morphisms will exist inside some fixed category with kernels.
\end{definition}

\begin{proposition}  \label{thm:restriction_kernel}
	Whenever we have morphisms $f: A \rightarrow B$ and $g: A \rightarrow C$, $\Ker(g\circ\ker(f)) \simeq \Ker(f\circ\ker(g))$
\end{proposition}

\begin{proof} We have this commutative diagram:

\begin{tikzpicture}
\matrix (m) [
matrix of math nodes,
row sep=5.5em,
column sep=6.5em,
text height=1.5ex, text depth=0.25ex
]
{ & & B & \\
	\Ker(f \circ \ker(g)) & \Ker(g) & A & C \\
	& \Ker(g \circ \ker(f)) & \Ker(f) & \\
};

\path[overlay,->, font=\scriptsize,>=latex]
(m-2-3) edge node[left] {$f$} (m-1-3)
(m-2-1) edge node[above] {$k_4 = \ker(f \circ \ker(g))$} (m-2-2)
(m-2-2) edge node[above] {$k_3 = \ker(g)$} (m-2-3)
(m-2-3) edge node[above] {$g$} (m-2-4)
(m-3-2) edge node[above] {$k_2 = \ker(g \circ \ker(f))$} (m-3-3)
(m-3-3) edge node[left] {$k_1 = \ker(f)$} (m-2-3);
\path[overlay,->, font=\scriptsize,>=latex,dashed]
(m-3-2) edge node[right] {$l$} (m-2-2)
(m-3-2) edge node[above] [bend right = 10] {$m$} (m-2-1)
(m-2-1) edge [bend right = 20] node[below] {$n$} (m-3-2);
\end{tikzpicture}

By definition, $g \circ k_1 \circ k_2 = 0$, and so $k_1 \circ k_2 = k_3 \circ l$ for some uniquely determined morphism $l$. Furthermore, $f \circ k_3 \circ l = f \circ k_1 \circ k_2 = 0$, and so $l = k_4 \circ m$ for some uniquely determined $m$. We then have $k_3 \circ k_4 \circ m = k_1 \circ k_2$. By a similar argument, we can find an $n:\Ker(f\circ \ker(g)) \rightarrow \Ker(g\circ \ker(f))$ with $k_1 \circ k_2 \circ n = k_3 \circ k_4$. Then $k_1 \circ k_2 \circ n \circ m = k_3 \circ k_4 \circ m = k_1 \circ k_2$. Since $k_1 \circ k_2$ is monic, $n \circ m= \text{id}_{\Ker(g\circ \ker(f))}$. By a similar argument, $m \circ n = \text{id}_{\Ker(f\circ \ker(g))}$. Therefore, $m$ and $n$ are isomorphisms.
\end{proof}

\begin{corollary} \label{thm:restriction_kernel2}
	Whenever we have morphisms $f: A \rightarrow B$ and $g: B \rightarrow C$, $\Ker(f) \simeq \Ker(f \circ \ker(g \circ f))$.
\end{corollary}

\begin{proof}
	Since $g \circ f \circ \ker(f) = 0$, we get from \ref{thm:zero_kernel} that ${\Ker(f) \simeq \Ker(g \circ f \circ \ker(f))}$. From \ref{thm:restriction_kernel}, $\Ker(g \circ f \circ \ker(f)) \simeq \Ker(f\circ \ker(g \circ f))$. The result follows from transitivity.
\end{proof}

\section{Ranked Categories}
\begin{definition} \label{def:ranked_cat}
	A \textbf{ranked category} is a category $\mathfrak{C}$ equipped with a \textbf{rank}, i.e. a function $\rk$ from the morphisms of $\mathfrak{C}$ to the nonnegative reals, which satisfies the following inequalities for any pair of morphisms $f: A \rightarrow B$ and $g: B \rightarrow C$ in $\mathfrak{C}$:
	\begin{enumerate}
		\item $\rk(g \circ f) \leq \rk(f)$
		\item $\rk(g \circ f) \leq \rk(g)$
	\end{enumerate}
	As an example, the ordinary matrix rank on the category of finite-dimensional vector spaces over some field $K$ is clearly a rank. An example on the category of finite sets is the function that measures the cardinality of the image of a morphism. In this section, we'll cover some basic definitions and theorems pertaining to ranked categories.
\end{definition}

\begin{definition} \label{def:ell}
	If $\mathfrak{C}$ is a ranked category, we'll use $\ell(A)$ as a shorthand for $\rk(\text{id}_A)$ whenever $A$ is an object of $\mathfrak{C}$.
\end{definition}

\begin{definition} \label{def:factors}
	If $f:A \rightarrow C$ is a morphism in $\mathfrak{C}$, and $g$ is another morphism in $\mathfrak{C}$,  we'll say that $f$ \textbf{factors through} $g$ if $f$ can be written as $g_1 \circ g \circ g_0$ for some pair of morphisms $g_1$ and $g_0$. If $B$ is an object in $\mathfrak{C}$, we'll say that $f$ \textbf{factors through} $B$ if it can be written as $h_1 \circ h_0$, with $h_0: A \rightarrow B$ and $h_1: B \rightarrow C$.
\end{definition}

The following theorems follow almost immediately from the axioms:

\begin{proposition} \label{thm:factors}
	If $f$ factors through a morphism $g$, then $\rk(f) \leq \rk(g)$.
\end{proposition}

\begin{proposition} \label{thm:factors2}
	If $f: A \rightarrow C$ factors through an object $B$, then $\rk(f) \leq \ell(B)$. As special cases, $\rk(f) \leq \ell(A)$ and $\rk(f) \leq \ell(C)$.
\end{proposition}

\begin{proposition} \label{thm:factors3}
	If $f$ factors through $g$ and $g$ factors through $f$, $\rk(f) = \rk(g)$.
\end{proposition}

The following are direct consequences of \ref{thm:factors3}:

\begin{proposition} \label{thm:isomorphic_ell}
	If $A$ and $B$ are isomorphic objects of $\mathfrak{C}$, then $\ell(A) = \ell(B)$.
\end{proposition}

\begin{proposition}  \label{thm:isomorphic_ell2}
	If $A$ and $B$ are isomorphic objects of $\mathfrak{C}$, with isomorphisms that commute with a pair of morphisms $f:A \rightarrow C$ and $g:B \rightarrow C$, then $\rk(f) = \rk(g)$.
\end{proposition}

\begin{corollary} \label{thm:restriction_rank_defined}
	If $f:A \rightarrow B$ and $g:A \rightarrow C$ are any morphisms in $\mathfrak{C}$, the quantity $\rk(f \circ \ker(g))$ is well-defined, i.e. it is independent of our choice of kernel for $g$.
\end{corollary}

\begin{proposition}  \label{thm:isomorphic_ell3}
	If $A$ and $B$ are isomorphic objects of $\mathfrak{C}$, with isomorphisms that commute with a pair of morphisms $f:C \rightarrow A$ and $g:C \rightarrow B$, then $\rk(f) = \rk(g).$
\end{proposition}

\begin{definition} \label{def:coh_exists2}
	Suppose $\mathfrak{C}$ is a ranked category with a zero object. Let $O$ be the objects of $\mathfrak{C}$ and $E$ be the class of triples $(A, B, C)$ for each short exact sequence $0 \rightarrow A \rightarrow B \rightarrow C \rightarrow 0$ in $\mathfrak{C}$. We'll say \textbf{cohomology exists on $\mathfrak{C}$} if cohomology exists for $(O, E, \ell)$, in the sense of \ref{def:coh_exists}. We then have a corresponding uniquely-determined sequence of functions $h = (h^0, h^1, \ldots)$, with $h^0 = \ell$.
\end{definition}

\begin{definition} \label{def:nullity}
	If $\mathfrak{C}$ is a ranked category with kernels, we'll use $\nul(f)$ as shorthand for $\ell(\Ker(f))$. We have from \ref{thm:isomorphic_ell} that this quantity is independent of our choice of kernel.
\end{definition}

\begin{proposition} \label{thm:zero_rank}
	If $g \circ f$ is a zero morphism, then $\rk(f) \leq \nul(g)$.
\end{proposition}
\begin{proof}
	We must have that $f$ factors through $\Ker(g)$, so applying \ref{thm:factors2} gives the result.
\end{proof}

\begin{definition} \label{def:coh_bullet}
	Suppose $\mathfrak{C}$ is a ranked category with kernels, and we have a chain complex $A$ in $\mathfrak{C}$ that's given by the following diagram:
	$$0 \xrightarrow{a_0} A_1 \xrightarrow{a_1} A_2 \xrightarrow{a_2}  A_3 \xrightarrow{a_3}\ldots$$	
	For $k \geq 0$, define $h^k_\bullet(A)$ to be $\nul(a_{k+1}) - \rk(a_k)$ (the dot indicates that we are looking at the cohomology of a resolution rather than an object; we'll use $h^k(L)$ later on for the cohomology of a lattice $L$). We have from \ref{thm:zero_rank} that these numbers are nonnegative.
\end{definition}

\begin{definition} \label{def:left_exact}
	We'll say that a rank on a category with kernels $\mathfrak{C}$ is \textbf{left-exact} if, for any morphism $f:A \rightarrow B$, we have $\rk(f) = \ell(A) - \nul(f)$. We'll say that a ranked category is left-exact if the corresponding rank is. As motivation for this definition, consider the following situation: suppose we have an Abelian category $\mathfrak{A}$ and a functor $F$ from  $\mathfrak{A}$ to the category $\normalfont\textbf{FVect}_K$ of finite dimensional vector spaces over some field $K$. Then for morphisms $f$ in $\mathfrak{A}$, defining $\rk(f)$ to be the rank of $F(f)$ makes $\mathfrak{A}$ a ranked category. If $F$ is a left-exact functor, then $\mathfrak{A}$ is clearly a left-exact ranked category.
\end{definition}

\begin{definition} \label{def:additive}
	Suppose $\mathfrak{C}$ is a left-exact ranked category with kernels. Suppose we have a sequence $A \xrightarrow{f} B \xrightarrow{g} C$ in $\mathfrak{C}$, with $f$ a kernel of $g$. Then we have from left-exactness and \ref{thm:factors2} that $\ell(C) - \ell(B) + \ell(A) = \ell(C) - \rk(g) \geq 0$. If we also have that $\ell(C) - \ell(B) + \ell(A) \leq \epsilon$, then we say that the sequence $A \xrightarrow{f} B \xrightarrow{g} C$ is \textbf{$\epsilon$-additive}. To motivate the terminology, consider that $\epsilon$-additivity implies $\ell(B) \approx \ell(A) + \ell(C)$ for small $\epsilon$.
\end{definition}

\begin{definition} \label{def:additive_object}
	Suppose $\mathfrak{C}$ is a left-exact ranked category with kernels. We say an object $A$ of $\mathfrak{C}$ is an \textbf{$\epsilon$-additive object} if, whenever we have a short exact sequence $0 \rightarrow A \xrightarrow{f} B \xrightarrow{g} C \rightarrow 0$, the sequence $A \xrightarrow{f} B \xrightarrow{g} C$ is $\epsilon$-additive.
\end{definition}

\section{Example: Coherent sheaves on a projective scheme over a field}
Suppose $K$ is a field, $X$ is a scheme that's projective over $K$, and $\Coh(X)$ is the category of coherent sheaves on $X$. The global sections of any sheaf in $\Coh(X)$ is a finite-dimensional $K$-vector space, and so taking global sections gives us a functor $\Gamma$ from $\Coh(X)$ to $\normalfont\textbf{FVect}_K$, where $\normalfont\textbf{FVect}_K$ is the category of finite-dimensional $K$-vector spaces. The usual matrix rank gives $\normalfont\textbf{FVect}_K$ a canonical ranked category structure, and the composite $\rk \circ \Gamma$ does the same for $\Coh(X)$. From the rank-nullity theorem, $\normalfont\textbf{FVect}_K$ is a left-exact ranked category. Since $\Gamma$ is a left-exact functor, it preserves kernels, and so $\Coh(X)$ inherits left-exactness from $\normalfont\textbf{FVect}_K$.

We'll now show that cohomology exists on $\Coh(X)$, in the sense of \ref{def:coh_exists2}. If $\mathcal{F}$ is a coherent sheaf on $X$, and $k \in \{0, 1, 2 \ldots\}$, let $h^k(\mathcal{F}) = \dim(H^k(X, \mathcal{F}))$. It suffices to show that $h := (h^0, h^1, h^2, \ldots)$ is an effaceable numerical $\delta$-functor.

\begin{lemma} \label{thm:numerical_delta}
	The sequence of functions $h$ on $\Coh(X)$ is a numerical $\delta$-functor.
\end{lemma}

\begin{proof}
	Whenever we have a short exact sequence in $\Coh(X)$, say ${0 \rightarrow \mathcal{F} \rightarrow \mathcal{G} \rightarrow \mathcal{H} \rightarrow 0}$, we have a long exact sequence
	$$0 \rightarrow H^0(X, \mathcal{F}) \rightarrow H^0(X, \mathcal{G}) \rightarrow H^0(X, \mathcal{H}) \rightarrow H^1(X, \mathcal{F}) \rightarrow H^1(X, \mathcal{G}) \ldots$$
	Recall that, when we have an exact sequence of vector spaces, taking the dimension of everything in the sequence yields a numerical exact sequence. In particular, taking the dimension of everything in the exact sequence above yields a numerical exact sequence
	$$h^0(\mathcal{F}), h^0(\mathcal{G}), h^0(\mathcal{H}), h^1(\mathcal{F}), h^1(\mathcal{G}), \ldots$$
	as required.
\end{proof}

\begin{lemma} \label{thm:effaceable}
	The sequence of functions $h$ is effaceable.
\end{lemma}

\begin{proof}
	It suffices to show that an arbitrary coherent sheaf $\mathcal{F}$ on $X$ can be embedded in an acyclic coherent sheaf.

	Since $X$ is projective over $K$, it admits a very ample line bundle $\mathcal{O}(1)$. Choose an integer $n$ such that $\mathcal{F}(n)$ is acyclic and $\mathcal{O}(n)$ is generated by global sections. Since $\mathcal{O}(n)$ is generated by global sections, there is a short exact sequence of the form $0 \rightarrow \mathcal{K} \rightarrow \mathcal{E} \rightarrow \mathcal{O}(n) \rightarrow 0$, where $\mathcal{E} $ is given by $\bigoplus_{i = 1}^N \mathcal{O}$ for some $N$. Since $\mathcal{O}(n)$ is locally projective, the short exact sequence $0 \rightarrow \mathcal{K} \rightarrow \mathcal{E} \rightarrow \mathcal{O}(n) \rightarrow 0$ must be locally split. Since split short exact sequences are preserved by additive functors, we can apply $\mathcal{H}om(-, \mathcal{O})$ and then $- \otimes \mathcal{F}(n)$ to obtain a locally split short exact sequence $0 \rightarrow \mathcal{F} \xrightarrow{f} \mathcal{E} \otimes \mathcal{F}(n) \rightarrow \mathcal{K}^\vee \otimes \mathcal{F}(n)\rightarrow 0$. Since $\mathcal{F}(n)$ is acyclic, so is $\mathcal{E} \otimes \mathcal{F}(n) = \bigoplus_{i = 1}^N \mathcal{F}(n)$. Therefore, $f$ is an embedding from $\mathcal{F}$ to an acyclic sheaf, as required.
\end{proof}

Putting together \ref{thm:numerical_delta} and \ref{thm:effaceable}, we have:

\begin{theorem} \label{thm:coh_exists}
	Cohomology exists on $\Coh(X)$.
\end{theorem}

\begin{remark}
	Our proof of \ref{thm:coh_exists} made use of the usual sheaf cohomology in an essential way. This is unsatisfactory in a way, since the whole point of numerical cohomology is to provide an approach to cohomology that doesn't require as much structure as more traditional approaches. It is known to the author, however, how \ref{thm:coh_exists} can be proven in a purely numerical way which parallels our approach to the cohomology of lattices given in section 10, and works for projective schemes of any dimension. This will possibly be the topic of a future paper.
\end{remark} 

\section{The Numerical Zig-Zag Lemma}

In this section, all objects and morphisms are in some fixed category $\mathfrak{C}$, which we assume to be a left-exact ranked category with kernels.

Suppose we have a commutative diagram as follows:

\begin{tikzpicture}
\matrix (m) [
matrix of math nodes,
row sep=2.5em,
column sep=2.5em,
text height=1.5ex, text depth=0.25ex
]
{ C_0 = 0 & C_1 & C_2 & C_3 & \ldots\\
	B_0 = 0 & B_1 & B_2 & B_3 & \ldots\\
	A_0 = 0 & A_1 & A_2 & A_3 & \ldots\\
};

\path[overlay,->, font=\scriptsize,>=latex]
(m-1-1) edge node[above] {$c_0$} (m-1-2)
(m-1-2) edge node[above] {$c_1$} (m-1-3)
(m-1-3) edge node[above] {$c_2$} (m-1-4)
(m-1-4) edge node[above] {$c_3$} (m-1-5)
(m-2-1) edge node[above] {$b_0$} (m-2-2)
(m-2-2) edge node[above] {$b_1$} (m-2-3)
(m-2-3) edge node[above] {$b_2$} (m-2-4)
(m-2-4) edge node[above] {$b_3$} (m-2-5)
(m-3-1) edge node[above] {$a_0$} (m-3-2)
(m-3-2) edge node[above] {$a_1$} (m-3-3)
(m-3-3) edge node[above] {$a_2$} (m-3-4)
(m-3-4) edge node[above] {$a_3$} (m-3-5)
(m-3-1) edge node[left] {$f_0$} (m-2-1)
(m-3-2) edge node[left] {$f_1$} (m-2-2)
(m-3-3) edge node[left] {$f_2$} (m-2-3)
(m-3-4) edge node[left] {$f_3$} (m-2-4)
(m-2-1) edge node[left] {$g_0$} (m-1-1)
(m-2-2) edge node[left] {$g_1$} (m-1-2)
(m-2-3) edge node[left] {$g_2$} (m-1-3)
(m-2-4) edge node[left] {$g_3$} (m-1-4);
\end{tikzpicture}

Suppose additionally that all rows are chain complexes and all columns are $\epsilon$-additive sequences (see \ref{def:additive}). We will call such a diagram an \textbf{$\epsilon$-additive sequence of chain complexes}.

Let $A$, $B$, and $C$ be the chain complexes corresponding to the rows of the diagram above, and let $(e_k)_{k=0}^\infty = (h_\bullet^0(A), h_\bullet^0(B), h_\bullet^0(C), h_\bullet^1(A), \ldots)$. The main goal of this section will be to prove that $(e_k)$ is almost a numerical exact sequence for small $\epsilon$; this can be thought of as an analogue of the zig-zag lemma. To be more precise, we will work up to the following theorem:

\begin{theorem}[The Numerical Zig-zag Lemma] \label{thm:zigzag}
	For all $k \geq 0$, we have $$\sum\limits_{j=0}^k (-1)^j e_{k-j} \geq -(k+1)\epsilon$$
\end{theorem}

One could use the ideas in this proof to get a stronger inequality than the one above, but this one suffices for our purposes (essentially we just need a bound that goes to $0$ as $\epsilon \rightarrow 0$) and is more convenient to prove.

The strategy will be to first find simpler expressions that approximate the alternating sums, and then to show that these simpler expressions can be bounded from below by a small negative constant, using the inequalities that $\rk$ obeys.

For $k \geq 0$, let $a_k$, $b_k$, and $c_k$ be given by:
$$\alpha_k = \nul(a_{k+1}) - \ell(A_k) + \nul(b_k) - \nul(c_k)$$
$$\beta_k = \nul(b_{k+1}) - \nul(a_{k+1}) - \rk(c_k)$$
$$\gamma_k = \nul(a_{k+1}) - \nul(b_{k+1}) + \nul(c_{k+1})$$

\begin{lemma} \label{thm:zigzag1}
	For all $k \geq 0$, we have:
	\begin{equation} \tag{a}
		\alpha_k + \beta_k - h_\bullet^k(B) \in [-\epsilon, 0]
	\end{equation}
	\begin{equation} \tag{b}
		\beta_k + \gamma_k - h_\bullet^k(C) = 0
	\end{equation}
	\begin{equation} \tag{c}
		\gamma_k + \alpha_{k + 1} - h_\bullet^{k+1}(A) = 0
	\end{equation}
\end{lemma}
\begin{proof}
	For (a), we immediately get the following from expanding the definitions and simplifying:
	$$\alpha_k + \beta_k - h_\bullet^k(B) = - \nul(c_k) - \rk(c_k) + \nul(b_k) + \rk(b_k) - \ell(A_k)$$
	Combining the above with the left-exactness of $\rk$ yields:
	$$\alpha_k + \beta_k - h_\bullet^k(B) = - \ell(C_k) + \ell(B_k) - \ell(A_k)$$
	and finally, since $A_k \rightarrow B_k \rightarrow C_k$ is $\epsilon$-additive, we must have that the right side is in $[-\epsilon, 0].$
	
	We get (b) immediately from expanding the definitions and simplifying.
	
	For (c), we immediately have $$\gamma_k + \alpha_{k + 1} - h_\bullet^{k+1}(A) = \nul(a_{k+1}) + \rk(a_{k+1}) - \ell(A_{k+1})$$ and the right side is $0$ from left-exactness.
\end{proof}
	
Let $(\phi_k)_{k=0}^\infty = (\alpha_0, \beta_0, \gamma_0, \alpha_1, \beta_1, \gamma_1, \alpha_2,\ldots)$. The previous lemma then yields, for all $k \geq 0$:
\begin{equation} \tag{1} \label{eqn:e_f}
	=\phi_k + \phi_{k+1} - e_{k+1} \in [-\epsilon, 0]
\end{equation}

This can be used to prove the following lemma, which essentially says that the alternating sums of $(e_k)$ are approximated by $(\phi_k)$ for small $\epsilon$:

\begin{lemma} \label{thm:zigzag2}
	For all $k \geq 0$, we have $\sum\limits_{j=0}^k (-1)^j e_{k-j} - \phi_k \in [-k\epsilon, k\epsilon]$.
\end{lemma}
\begin{proof}
	This holds when $k = 0$, since $e_0 = \phi_0 = \nul(a_1)$. Suppose the result holds for some arbitrary $k$. Then we have 
	$$\sum\limits_{j=0}^k (-1)^j e_{k-j} - \phi_k \in [-k\epsilon, k\epsilon]$$
	Adding \eqref{eqn:e_f} yields:
	$$\sum\limits_{j=0}^k(-1)^j e_{k-j} + \phi_{k+1} - e_{k+1} \in [-(k+1)\epsilon, k\epsilon]$$
	This is clearly equivalent to:
	$$\sum\limits_{j=0}^{k+1} (-1)^j e_{k+ 1 -j} - \phi_{k+1} \in [-k\epsilon, (k+1)\epsilon]$$
	and so the result holds for k+1. Therefore, by induction, it must hold for all $k$.
\end{proof}

Now that we have that the alternating sums of $(e_k)$ are approximated by $(\phi_k)$, we can use bounds on the latter to obtain bounds on the former. To find bounds on the $(\phi_k)$, we'll find even simpler expressions that approximate the $\alpha_k$, $\beta_k$, and $\gamma_k$, and then use basic properties of ranked categories to bound these simpler expressions. First, however, we'll need a few technical lemmas.

\begin{lemma} \label{thm:zigzag3}
	For all $k \geq 0$, $\Ker(f_{k+1} \circ a_k) \simeq \Ker(a_k)$
\end{lemma}
\begin{proof}
	This follows from \ref{thm:kernel_monic} and \ref{thm:kernel_compose}.
\end{proof}

\begin{lemma} \label{thm:zigzag4}
	For all $k \geq 0$, $\rk(g_k \circ \ker(b_k)) = \nul(b_k) - \nul(a_k)$
\end{lemma}
\begin{proof}
	From \ref{thm:restriction_kernel}, we have
	$$\Ker(g_k\circ\ker(b_k)) \simeq \Ker(b_k\circ\ker(g_k)) \simeq \Ker(b_k \circ f_k)$$
	Combining the above with commutativity and \ref{thm:zigzag3}, we have
	$$\Ker(g_k\circ \ker(b_k)) \simeq \Ker(f_{k+1} \circ a_k) \simeq \Ker(a_k)$$
	From left-exactness, we have
	$$\rk(g_k\circ \ker(b_k)) = \ell(\Ker(b_k)) - \ell(\Ker(g_k\circ\ker(b_k)))$$
	$$= \nul(b_k) - \ell(\Ker(a_k))$$
	$$= \nul(b_k) - \nul(a_k)$$
\end{proof}
\begin{lemma} \label{thm:zigzag5}
	For all $k \geq 0$, $\rk(b_k\circ\ker(g_{k+1} \circ b_k)) = \nul(g_{k+1} \circ b_k) - \nul(b_k)$
\end{lemma}
\begin{proof}
	We have from \ref{thm:restriction_kernel2} that $\Ker(b_k\circ\ker(g_{k+1} \circ b_k)) \simeq \Ker(b_k)$, the result then follows from left-exactness.
\end{proof}

\begin{lemma} \label{thm:zigzag6}
	For all $k \geq 0$, $\rk(g_k\circ\ker(c_k \circ g_k)) = \nul(c_k \circ g_k) - \ell(A_k)$
\end{lemma}

\begin{proof}
	From \ref{thm:restriction_kernel2}, we have
	$$\Ker(g_k\circ \ker(c_k \circ g_k)) \simeq \Ker(g_k) \simeq A_k$$
	This, together with the left-exactness axiom, yields the result.
\end{proof}

For $k \geq 0$, define $x_k$ and $y_k$ by
$$x_k = \nul(g_{k+1} \circ b_k) - \ell(A_k) - \nul(c_k)$$
$$y_k = \rk(g_{k+1} \circ b_k) - \rk(c_k)$$
We then have the following:

\begin{lemma} \label{thm:zigzag7}
	For all $k \geq 0$,
	\begin{equation} \tag{a}
		x_k + y_k \in [-\epsilon, 0]
	\end{equation}
	\begin{equation} \tag{b}
		x_k \leq 0
	\end{equation}
	\begin{equation} \tag{c}
		y_k \leq 0
	\end{equation}
	\begin{equation}
		x_k \geq -\epsilon \tag{d}
	\end{equation}
	\begin{equation}
		y_k \geq -\epsilon \tag{e}
	\end{equation}
\end{lemma}
\begin{proof}
	For (a), we get the following just by rearranging terms:
	$$x_k + y_k = \rk(g_{k+1} \circ b_k) +\nul(g_{k+1} \circ b_k) - \rk(c_k) - \nul(c_k) - \ell(A_k)$$
	Using the left-exactness axiom twice yields:
	$$x_k + y_k = \ell(B_k) - \ell(C_k) - \ell(A_k)$$
	We know the right side is in $[-\epsilon, 0]$ from $\epsilon$-additivity, and this yields (a).
	
	For (b), clearly $c_k \circ g_k\circ \ker(c_k \circ g_k) = 0$, so by \ref{thm:zero_rank}, $\rk(g_k\circ \ker(c_k \circ g_k)) \leq \nul(c_k)$. Combining this with \ref{thm:zigzag6} yields
	$$\nul(c_k \circ g_k) - \ell(A_k) - \nul(c_k) \leq 0$$
	From commutativity, we can substitute $g_{k+1} \circ b_k$ for $c_k \circ g_k$ above, which gives the result.
	
	We get (c) from commutativity and the second rank axiom:
	$$y_k = \rk(g_{k+1} \circ b_k) - \rk(c_k) = \rk(c_k \circ g_k) - \rk(c_k) \leq 0$$
	
	We get (d) from combining (c) and (a). Similarly, we get (e) from combining (b) and (a).
\end{proof}

We now have the tools to establish these approximations of $\alpha_k$, $\beta_k$, and $\gamma_k$:

\begin{lemma} \label{thm:zigzag8}
	For all $k \geq 0$,
	\begin{equation} \tag{a}
		\alpha_k - \nul(a_{k+1}) + \rk(b_k\circ \ker(g_{k+1} \circ b_k)) \in [-\epsilon, 0]
	\end{equation}
	\begin{equation} \tag{b}
		\beta_k - \rk(g_{k+1}\circ \ker(b_{k+1})) + \rk(g_{k+1} \circ b_k) \in [-\epsilon, 0]
	\end{equation}
	\begin{equation} \tag{c}
		\gamma_k - \nul(c_{k+1}) + \rk(g_{k+1}\circ \ker(b_{k+1})) = 0
	\end{equation}
\end{lemma}
\begin{proof}
	For (a), we get the following from \ref{thm:zigzag5} and some simple arithmetic manipulation: 
	$$\alpha_k - \nul(a_{k+1}) + \rk(b_k\circ \ker(g_{k+1} \circ b_k)) = x_k$$
	The result then follows from \ref{thm:zigzag7}.
	
	For (b), we get the following from \ref{thm:zigzag4} and some simple arithmetic manipulation:
	$$\beta_k - \rk(g_{k+1}\circ \ker(b_{k+1})) + \rk(g_{k+1} \circ b_k) = y_k$$
	The result then follows from \ref{thm:zigzag7}.
	
	We get (c) directly from \ref{thm:zigzag4} and some simple arithmetic manipulation.
\end{proof}

Now we'll put bounds on the expressions that approximate the $\alpha_k$, $\beta_k$, and $\gamma_k$.

\begin{lemma} \label{thm:zigzag9}
	For all $k \geq 0$,
	\begin{equation} \tag{a}
		\nul(a_{k+1}) - \rk(b_k\circ \ker(g_{k+1} \circ b_k)) \geq 0
	\end{equation}
	\begin{equation} \tag{b}
		\rk(g_{k+1}\circ \ker(b_{k+1})) - \rk(g_{k+1} \circ b_k) \geq 0
	\end{equation}
	\begin{equation} \tag{c}
		\nul(c_{k}) - \rk(g_k\circ \ker(b_k)) \geq 0
	\end{equation}
\end{lemma}
\begin{proof}
For (a), since $f_{k+1}$ is a kernel of $g_{k+1}$, and $g_{k+1} \circ b_k \circ \ker(g_{k+1} \circ b_k) = 0$, we have that $b_k \circ \ker(g_{k+1} \circ b_k)$ factors uniquely through $f_{k+1}$; we'll say $b_k \circ \ker(g_{k+1} \circ b_k) = f_{k+1} \circ q_k$. From commutativity and the fact that $B$ is a chain complex, we have

$$f_{k+2} \circ a_{k+1} \circ q_k = b_{k+1} \circ f_{k+1} \circ q_k = b_{k+1} \circ b_k \circ \ker(g_{k+1} \circ b_k) = 0$$

Therefore, we can apply \ref{thm:zero_rank} to get that $\rk(q_k) \leq \nul(f_{k+2} \circ a_{k+1})$. Applying \ref{thm:zigzag3} then yields $\rk(q_k) \leq \nul(a_{k+1})$. Finally, we have:
$$\rk(b_k \circ \ker(g_{k+1} \circ b_k)) = \rk(f_{k+1} \circ q_k) \leq \rk(q_k) \leq \nul(a_{k+1})$$
which gives us (a).

For (b), since $B$ is a chain complex, we have that $b_k$ factors through $\ker(b_{k+1})$, and so $g_{k+1} \circ b_k$ factors through $g_{k+1}\circ \ker(b_{k+1})$. The result then follows from the second rank axiom.

For (c), we have from commutativity that
$$c_k \circ g_k\circ \ker(b_k) = g_{k+1} \circ b_k\circ \ker(b_k)= 0$$
Therefore, we can apply \ref{thm:zero_rank} to get the result.
\end{proof}

Combining \ref{thm:zigzag8} and \ref{thm:zigzag9}, we have $\alpha_k \geq -\epsilon$, $\beta_k \geq -\epsilon$, and $\gamma_k \geq 0$ for all $k \geq 0$. These three facts together yield:

\begin{lemma} \label{thm:zigzag10}
	For all $k \geq 0$, $\phi_k \geq - \epsilon$.
\end{lemma}

Now we are finally in a position to prove the main theorem.

\begin{proof} [Proof of Theorem \ref{thm:zigzag}]
	From \ref{thm:zigzag2}, we have $\sum\limits_{j=0}^k (-1)^j e_{k-j} - \phi_k \geq -k\epsilon$. Adding this to the inequality of \ref{thm:zigzag10} yields the result.
\end{proof}

\section{Lattice Preliminaries}

\begin{definition} \label{def:lattice}
	We'll say a subgroup $L$ of a finite-dimensional inner product space $V$ is a \textbf{lattice} if it satisfies either of these equivalent conditions:
	\begin{enumerate}
		\item $L$ is generated by a linearly independent set of vectors.
		\item $L$ is a discrete subset of $V$.
	\end{enumerate}
	The equivalence of these conditions is proven in \cite[Chapter 1, Part 4]{Neukirch}.
\end{definition}

\begin{definition} \label{def:local-dimension}
	We'll say a lattice $L$ is \textbf{locally d-dimensional} if $\spn(L)$ is a $d$-dimensional vector space. The usual terminology is to say $L$ is a lattice of rank $d$, but we are already using the term ``rank'' to refer to something else.
\end{definition}

\begin{definition} \label{def:basis}
	A \textbf{basis} of a lattice $L$ is a set of linearly independent vectors $\{x_1 \ldots x_d\}$ that generate $L$. By definition, such a set must exist for any lattice. Given a choice of basis for $L$, we say a \textbf{fundamental mesh} of $L$ is the set of all elements of $\spn(L)$ which can be written as a linear combination of basis elements, with all coefficients in $[0, 1]$. A fundamental mesh $F$ of $L$ clearly has the following useful property: if $x$ is any point of $\spn(L)$, the set $x + F$ contains a point of $L$.
\end{definition}

\begin{proposition} \label{thm:f_R}
	Suppose $L$ and $M$ are lattices, and $f: L \rightarrow M$ is a group homomorphism. Then there exists a unique linear map $f_\mathbb{R}: \spn(L) \rightarrow \spn(M)$ such that $f_\mathbb{R}(x) = f(x)$ for $x \in L$. Furthermore, assuming the following definitions of $||f||$ and $||f_\mathbb{R}||$:
	$$||f(x)|| := \sup_{x \in L} \frac{||f(x)||}{||x||}$$
	$$||f_\mathbb{R}(x)|| := \sup_{x \in \spn(L)} \frac{||f_\mathbb{R}(x)||}{||x||},$$	
	we have $||f_\mathbb{R}|| = ||f||$.
\end{proposition}

\begin{proof}
	The existence and uniqueness of $f_\mathbb{R}$ is trivial if we use the first definition of lattices. We'll prove now that $||f|| = ||f_\mathbb{R}||$. Let $F$ be a fundamental mesh of $L$. Let $r_1 = \sup_{x \in F} ||x||$ and $r_2 = \sup_{x \in F} ||f(x)||$. We clearly have that $||f_\mathbb{R}|| \geq ||f||$, so either $||f_\mathbb{R}|| > ||f||$ or $||f_\mathbb{R}|| = ||f||$. In the first case, we can find a $y \in \spn(L)$ with $\frac{||f_\mathbb{R}(y)||}{||y||} > ||f||$. We then have for a sufficiently small positive $\lambda \in \mathbb{R}$ that $\frac{||f_\mathbb{R}(y)|| - \lambda r_1 }{||y|| + \lambda r_2} > ||f||$. Let $l$ be a point of $L$ inside $\frac{y}{\lambda} + F$. From the triangle inequality, and the previous inequality, we have:
	$$\frac{||f(l)||}{||l||} \geq \frac{||f_\mathbb{R}(y)|| - \lambda r_1 }{||y|| + \lambda r_2} > ||f||,$$
	which is a contradiction. Therefore $||f_\mathbb{R}|| = ||f||$.
\end{proof}

\begin{definition} \label{def:homomorphism}
	If $L$ and $M$ are lattices, and $f:L \rightarrow M$ is a group homomorphism, we say that $f$ is a \textbf{lattice homomorphism} if $||f|| \leq 1$, or equivalently, if $||f_\mathbb{R}|| \leq 1$.
\end{definition}

\begin{definition} \label{def:lat_cat}
	Define the category of lattices, $\Lat$, to be the category whose objects are lattices and whose morphisms are lattice homomorphisms. This category clearly has a zero object: the one-element lattice containing just a zero. It also has kernels and cokernels, which will be characterized shortly. Also, clearly two lattices are isomorphic in $\Lat$ iff there is an isometric group isomorphism between them.
\end{definition}

\begin{definition} \label{def:sublattice}
	If $L$ is a lattice, a \textbf{sublattice} of L is a subset of $L$ that is also a lattice. A \textbf{normal sublattice} is a sublattice of the form $L \cap V$, where $V$ is a subspace of the vector space $\spn(L)$. We'll say that a lattice homomorphism is a \textbf{normal embedding} if it's an isometric embedding and its image is a normal sublattice of its codomain.
\end{definition}

\begin{proposition} \label{thm:projection_sublattice}
	Suppose $L$ is a lattice with a sublattice $K$. Let $P$ be the projection operator from $\spn(L)$ to $\spn(K)^\bot$. Then $P(L)$ is a lattice.
\end{proposition}
\begin{proof}
	Clearly $P(L)$ is a subgroup of $\spn(L)$. It remains to show that $P(L)$ is discrete.
	Let $B$ be the $1$-ball in $\spn(K)^\bot$, and let $F$ be a fundamental mesh of $K$. For each $v \in P(L)$, there is clearly a lattice point of $L$ somewhere in $v + F$. Therefore, the projection map from $L \cap (B + F)$ to $P(L) \cap B$ is surjective. However, there are only finitely many points in $L \cap (B + F)$, since it's the intersection of a discrete set with a compact set. Therefore, there are only finitely many points in $P(L) \cap B$. Since we already know $P(L)$ is a group, this last fact implies it is discrete.
\end{proof}

\begin{definition} \label{def:quotient_lattice}
	Suppose that $L$ is a lattice and $K$ is a normal sublattice of $L$. We'll denote the projection of $L$ onto $\spn(K)^\bot$ as $L/K$. We have from \ref{thm:projection_sublattice} that $L/K$ is a lattice. We'll call a lattice constructed this way a \textbf{quotient lattice}. 
\end{definition}

The next proposition is trivial and so the proof will be skipped.

\begin{proposition} \label{thm:kernels}
	Suppose $f: L \rightarrow M$ is a lattice homomorphism. Let $K$ be the group-theoretic kernel of $f$, i.e. $K = \{x \in L | f(x) = 0\}$. Then $K$ is a normal sublattice of $L$, and the inclusion map $i: K \rightarrow L$ is a kernel of $f$ in $\Lat$.
\end{proposition}

\begin{proposition} \label{thm:cokernels}
	Suppose $f: L \rightarrow M$ is a lattice homomorphism. Let $N = M \cap \spn(f(L))$, which is clearly a normal sublattice of $M$. Then the projection $g: M \rightarrow M/N$ is a cokernel of $f$ in $\Lat$.
\end{proposition}
\begin{proof}
	Suppose we have a lattice homomorphism $h: M \rightarrow O$ such that $h \circ f = 0$. We need to show that $h$ factors uniquely as $i \circ g$ for some homomorphism $i: M/N \rightarrow O$. The uniqueness follows directly from the surjectivity of $g$: clearly $i$ must map $g(x)$ to $h(x)$. To show that such an $i$ exists, it suffices to show that $||h(x)|| \leq ||g(x)||$ for any $x \in M$. Since $h \circ f = 0$, we clearly have that $h_\mathbb{R} \circ f_\mathbb{R} = 0$, and so $h_\mathbb{R}(x - g(x)) = 0$, and so $h(x) = h_\mathbb{R}(g(x))$. But, since $h$ is a lattice homomorphism, we have $||h_\mathbb{R}(g(x))|| \leq ||g(x)||$. Altogether, we have $||h(x)|| = ||h_\mathbb{R}(g(x))|| \leq ||g(x)||$, as required.
\end{proof}

We have from the above theorems that kernels and cokernels always exist in $\Lat$. Furthermore, since kernels and cokernels in general are uniquely determined up to unique commuting isomorphism, the above theorems give a characterization of all kernels and cokernels on $\Lat$. We also have that, whenever $K$ is a normal sublattice of $L$, there is a short exact sequence $0 \rightarrow K \rightarrow L \rightarrow L/K \rightarrow 0$. Furthermore, any short exact sequence in $\Lat$ is isomorphic to one obtained this way.

\begin{definition} \label{def:ses_ips}
	We say a sequence of inner product spaces $0 \rightarrow U \xrightarrow{f} V \xrightarrow{g} W \rightarrow 0$ is a \textbf{short exact sequence of inner product spaces} if $f^* f = \unit_U$, $gg^* = \unit_W$, and $f f^* + g^* g = \unit_V$. In other words, the sequence ``looks like'' $0 \rightarrow U \rightarrow U \oplus W \rightarrow W \rightarrow 0$, with the morphisms given by the canonical embedding and projection.
\end{definition}

\begin{proposition} \label{thm:exactness_local}
	A diagram of the form $0 \rightarrow L \xrightarrow{f} M \xrightarrow{g} N \rightarrow 0$ is a short exact sequence iff it induces a short exact sequence of the underlying abelian groups, and the sequence $0 \rightarrow \spn(L) \rightarrow \spn(M) \rightarrow \spn(N) \rightarrow 0$ is a short exact sequence of inner product spaces.
\end{proposition}

\begin{proof}
	If $0 \rightarrow L \rightarrow M \rightarrow N \rightarrow 0$ is a short exact sequence of lattices, then we clearly have the latter properties from our characterization of kernels and cokernels above.
	
	Conversely, suppose the $0 \rightarrow L \rightarrow M \rightarrow N \rightarrow 0$ satisfies the latter properties. We know  $0 \rightarrow \Ker(g) \rightarrow M \rightarrow M/\Ker(g) \rightarrow 0$ is a short exact sequence, so the same must be true of $0 \rightarrow L \rightarrow M \rightarrow N \rightarrow 0$ if we can find appropriately commuting isomorphisms between these two sequences. The required isomorphisms between $L$ and $\Ker(g)$ are provided by $f$ and $f^*$; the required isomorphisms between $N$ and $M/\Ker(g)$ are provided by $g^*$ and $g$. This completes the proof.
\end{proof}

\begin{remark}
	The above can be thought of as a local condition for exactness: inducing a short exact sequence of the underlying abelian groups is like being exact at each finite prime, and inducing a short exact sequence of the underlying inner product spaces is like being exact at the prime at infinity.
\end{remark}

\begin{definition} \label{def:tensor_objects}
	Suppose $V$ and $W$ are inner product spaces. We define the tensor product $V \otimes W$ to be the usual tensor product of vector spaces, equipped with the inner product $\langle v_1 \otimes w_1 , v_2 \otimes w_2 \rangle = \langle v_1 , v_2 \rangle \langle w_1 , w_2 \rangle$. 
	
	Suppose $L$ and $M$ are lattices. We define $L \otimes M$ to be the subgroup of the inner product space $\spn(L) \otimes \spn(M)$ generated by tensors of the form $x \otimes y$, with $x \in L$ and $y \in M$. This subgroup must be a lattice, for if we have bases $B_L$ and $B_M$ for $L$ and $M$, then $L \otimes M$ is clearly generated by the linearly independent set $\{ x \otimes y | x \in B_L, y \in B_M \}$.
\end{definition}

\begin{definition} \label{def:tensor_morphisms}
	Whenever we have a linear map between inner product spaces, say $f: V \rightarrow W$, and $U$ is any inner product space, we define ${U \otimes f : U \otimes V \rightarrow U \otimes W}$ to be the linear map taking $u \otimes v$ to $u \otimes f(v)$. We define $f \otimes U$ similarly. When we have lattices $L$, $M$, and $N$, and a homomorphism $g: L \rightarrow M$, we'll assume the obvious analogous definition for $N \otimes g$ and $g \otimes N$.
\end{definition}

\begin{proposition} \label{thm:tensor_adjoint}
	Suppose $f: V \rightarrow W$ is a linear map between inner product spaces, and $U$ is any inner product space. Then $(U \otimes f)^* = U \otimes f^*$, and $(f \otimes U)^* = f^* \otimes U$.
\end{proposition}
\begin{proof}
	We have:
	$$\langle (U \otimes f)(u_1 \otimes v) , u_2 \otimes w \rangle =
	\langle u_1 \otimes f(v) , u_2 \otimes w \rangle $$
	$$= \langle u_1 , u_2 \rangle \langle f(v), w \rangle $$
	$$= \langle u_1 , u_2 \rangle \langle v, f^*(w) \rangle $$
	$$ = \langle u_1 \otimes v , u_2 \otimes f^*(w) \rangle $$
	$$= \langle u_1 \otimes v, (U \otimes f^*)(u_2 \otimes w) \rangle $$
	So, we have from bilinearity that $\langle (U \otimes f)(x), y \rangle = \langle x, (U \otimes f^*)(y) \rangle$ for all $x \in U \otimes V$ and $y$ in $U \otimes W$. This proves that $(U \otimes f)^* = U \otimes f^*$. By symmetry, we have a similar proof that $(f \otimes U)^* = f^* \otimes U$.
\end{proof}

\begin{proposition} \label{thm:tensor_exact}
	Suppose $0 \rightarrow K \rightarrow L \rightarrow M \rightarrow 0$ is a short exact sequence of lattices, and $N$ is any lattice. Then the following are both short exact sequences: 
	$$0 \rightarrow K \otimes N \rightarrow L \otimes N \rightarrow M \otimes N  \rightarrow 0$$
	$$0 \rightarrow N \otimes K \rightarrow N \otimes L \rightarrow N \otimes M \rightarrow 0$$
\end{proposition}

\begin{proof}
	Our strategy will be to show that $0 \rightarrow K \otimes N \rightarrow L \otimes N \rightarrow M \otimes N  \rightarrow 0$ satisfies both conditions of \ref{thm:exactness_local}. Clearly, tensoring with a free abelian group preserves short exact sequences of abelian groups. This gives the first condition. The second condition follows from \ref{thm:tensor_adjoint}. Therefore, \ref{thm:exactness_local} implies that the first sequence is short exact. Clearly, by symmetry, a similar argument shows that the second sequence is short exact as well.
\end{proof}

\begin{definition} \label{def:volume}
	Suppose $L$ is a lattice with basis $x_1 \ldots x_d$. Let $e_1 \ldots e_d$ be an orthonormal basis of $\spn(L)$. Let $T$ be the unique operator on $\spn(L)$ satisfying $T(e_i) = x_i$ for $1 \leq i \leq d$. We define the \textbf{volume} of $L$, $\vol(L)$, to be $|\det(T)|$. It can be shown that this quantity is independent of our choice of $x_i$ and $e_i$ \cite[Chapter 1, Part 4]{Neukirch}. Furthermore, two isomorphic lattices clearly have the same volume. 
\end{definition}

\begin{proposition} \label{thm:exact_volume}
	Suppose $0 \rightarrow L \xrightarrow{f} M \xrightarrow{g} N \rightarrow 0$ is a short exact sequence of lattices. Then $\vol(M) = \vol(L) \cdot \vol(N)$.
\end{proposition}

\begin{proof}
	Since isomorphic lattices have the same volume, and two kernels of a given morphism are isomorphic, we can assume without loss of generality that $L$ is a normal sublattice of $M$, with $f$ the inclusion morphism, by \ref{thm:kernels}. Similarly, we can assume without loss of generality that $N = M/L$, with $g$ the orthogonal projection, by \ref{thm:cokernels}. Let $x_1 \ldots x_c$ be a basis of $L$, and $y_1 \ldots y_d$ be a basis of $N$. Choose $x_{c+1} \ldots x_{c+d}$ in $M$ so that $g(x_{c+i}) = y_i$. Then $x_1 \ldots x_{c + d}$ is a basis for $M$. Let $e_1 \ldots e_c$ be an orthonormal basis for $\spn(L)$, and $e_{c+1} \ldots e_{c+d}$ be an orthonormal basis for $\spn(N) = \spn(L)^\bot$. Then $e_1 \ldots e_{c+d}$ is an orthonormal basis for $\spn(M)$. Let $T$ be the unique linear operator on $\spn(M)$ satisfying $T(e_i) = x_i$ for $1 \leq i \leq c + d$. We clearly have that $\spn(L)$ is a $T$-invariant subspace of $\spn(M)$, so $T$ induces restriction and quotient operators on $\spn(L)$ and $\spn(M) / \spn(L)$. The latter then induces an operator on $\spn(N)$, since we have a canonical isomorphism between $\spn(N) = \spn(L)^{\perp}$ and $\spn(M) / \spn(L)$. We'll denote these operators on $\spn(L)$ and $\spn(N)$ as $T_L$ and $T_N$. We then have $\det(T) = \det(T_L) \cdot \det(T_N)$. Taking the absolute value of both sides yields the result.
\end{proof}

\section{The Canonical Rank on $\Lat$}

\begin{definition}
Given an element $x$ of an inner product space, let $\rho(x)$ denote $e^{-\pi||x||^2}$. Given a subset $S$ of an inner product space, let $\rho(S)$ denote $\sum_{x \in S}{\rho(x)}$.
\end{definition}

Suppose $f: \mathcal{F} \rightarrow \mathcal{G}$ is a morphism between coherent sheaves on a projective scheme over a field, and $\rk(f)$ is the rank of the corresponding linear map between global sections. Then we clearly have $\rk(f) = \ell(\mathcal{F}) - \ell(\Ker(f))$, from the rank-nullity theorem and the fact that the global sections functor preserves kernels.

Suppose $L$ is a lattice. Following \cite{VDGS}, we think of $L$ as like a sheaf on the ``completion'' of $\Spec(Z)$, and think of $\ell(L) := \ln(\rho(L))$ as like the dimension of the global sections of $L$. Then, given a morphism of lattices, say $f: L \rightarrow M$, the previous paragraph suggests that we think of $\ell(L) - \ell(\Ker(f))$ as like the rank of the induced map between global sections. We therefore define $\rk(f)$ as $\ell(L) - \ell(\Ker(f))$. This gives us a function $\rk$ from the morphisms of $\Lat$ to the non-negative real numbers. The purpose of this section is to show that $\rk$ is a left-exact rank, as we should expect from the analogies discussed so far. This makes $\Lat$ a left-exact ranked category with kernels. We'll show in subsequent sections, using the numerical zig-zag lemma, that cohomology exists on $\Lat$ (in the sense of \ref{def:coh_exists2}) and agrees with the definition given in \cite{VDGS}.

It follows directly from the above definition of $\rk$ that if it is a rank, it must be left-exact. Also, if we have homomorphisms $f: L \rightarrow M$ and $g: M \rightarrow N$, we must have that $\rk(g \circ f) \leq \rk(f)$, since $\Ker(g \circ f)$ contains $\Ker(f)$. It remains to show that $\rk(g \circ f) \leq \rk(g)$. This hinges upon the following recent result of Oded Regev and Noah Stephens-Davidowitz:

\begin{proposition} \label{thm:rho_inequality}
	If $L$ is a lattice with sublattices $M$ and $N$, we have
	$$\frac{\rho(M)}{\rho(M \cap N)} \leq \frac{\rho(L)}{\rho(N)}$$
\end{proposition}

\begin{proof}
	A nearly identical equivalent statement is proven as Theorem 5.1 of \cite{RSD}.
\end{proof}

\begin{definition} \label{def:strictly_contracting}
	We'll say a lattice homomorphism $f$ is \textbf{strictly contracting} if $||f|| < 1$, or equivalently, if $||f_\mathbb{R}|| < 1$ (this notation is explained in \ref{thm:f_R}).
\end{definition}

We'll show that $\rk$ obeys the second rank axiom by first proving it for strictly contracting homomorphisms.

\begin{lemma} \label{thm:lattice_rk_lemma}
	Suppose we have lattice homomorphisms $f: L \rightarrow M$ and $g: M \rightarrow N$. Suppose furthermore that $f$ is strictly contracting. Then $\rk(g \circ f) \leq \rk(g)$.
\end{lemma}
\begin{proof}
	Let V be the vector space $\spn(L)$. In addition to the original inner product on $V$, we have a second symmetric bilinear form defined as $\langle \langle x, y \rangle \rangle := \langle x, y \rangle - \langle f_\mathbb{R}(x), f_\mathbb{R}(y) \rangle$. Since $f$ is strictly contracting, we also have that this form is positive-definite, and is therefore an inner product. Let $L'$ be the lattice consisting of the same underlying subgroup of the same vector space as $L$, but with this new inner product. Let $O$ be the orthogonal direct sum $L' \oplus M$. Let $P$ denote the sublattice of $O$ corresponding to the graph of $f$. It's easy to verify that $P$ is isometrically isomorphic to $L$. Let $K$ be the subspace of $O$ consisting of points whose right coordinate lies in $\Ker(g)$. Applying \ref{thm:rho_inequality}, we have
	$$\frac{\rho(P)}{\rho(P \cap K)} \leq \frac{\rho(O)}{\rho(K)}.$$
	Equivalently,
	$$\frac{\rho(L)}{\rho(\Ker(g \circ f))} \leq \frac{\rho(L')\rho(M)}{\rho(L')\rho(\Ker(g))}$$
	
	The result then follows from simplifying and applying $\ln$ to both sides. 
\end{proof}

We can now finish the proof that $\rk$ is a left-exact rank.

\begin{theorem} \label{thm:lattice_rank_compose}
	Suppose we have lattice homomorphisms $f: L \rightarrow M$ and $g: M \rightarrow N$. Then $\rk(g \circ f) \leq \rk(g)$.
\end{theorem}

\begin{proof}
	For any positive real number $\epsilon$, let $L_\epsilon$ be the lattice $\{(1 + \epsilon)x | x \in L\}$. Dividing by $1 + \epsilon$ gives a lattice homomorphism from $L_\epsilon$ to $L$, and composing with $f$ gives a lattice homomorphism $f_\epsilon: L_\epsilon \rightarrow M$. Clearly $f_\epsilon$ is strictly contracting, and so we have from \ref{thm:lattice_rk_lemma} that $\rk(g \circ f_\epsilon) \leq \rk(g)$. Since this holds for all $\epsilon > 0$, we have from continuity that $\rk(g \circ f) \leq \rk(g)$.
\end{proof}

\begin{corollary} \label{thm:lattice_rank_complete}
	$\rk$ is a left-exact rank.
\end{corollary}

\section{Tight Lattices}

In this section, we'll define tight lattices and establish some properties about them that we'll need later. Tightness is a sort of positivity condition for lattices. We'll see shortly that it's stable under tensor products. We'll also see, once we have defined the cohomology of lattices, that tight lattices have small $h^1$. In the next section, we'll show that cohomology exists on $\Lat$ by defining the numerical cohomology of a lattice in terms of the numerical cohomology of a tight resolution. This is motivated by the fact that, in general, cohomology can be computed with acyclic resolutions.

\begin{definition}
	We say a lattice $L$ of local dimension $d$ is \textbf{$\delta$-tight} if it's generated by vectors of norm at most $\frac{\delta}{d^2}$. We'll informally say that $L$ is \textbf{tight} if it's $\delta$-tight for small $\delta$.
\end{definition}

\begin{remark}
	The reader might find it interesting to note the similarity between $\delta$-tightness and the property, for actual sheaves of modules, of being generated by global sections.
\end{remark}

\begin{proposition} \label{thm:tensor_tight}
	If a lattice $L$ is $\delta$-tight, and $M$ is $\gamma$-tight, then $L \otimes M$ is $\delta\gamma$-tight.
\end{proposition}

\begin{proof}
	Let $d_L$ and $d_M$ be the local dimensions of $L$ and $M$ respectively. Clearly $L \otimes M$ is locally $d_Ld_M$-dimensional, so we need to show that it can be generated by vectors of norm at most $\frac{\delta\gamma}{d_L^2d_M^2}$. We have that $L$ is generated by some set of vectors $S_L$ with norms at most $\frac{\delta}{d_L^2}$, and $M$ is generated by some set of vectors $S_M$ with norms at most $\frac{\gamma}{d_M^2}$. Then $\{x \otimes y | x \in S_L, y \in S_M\}$ clearly generates $L \otimes M$, and consists of vectors of norm at most $\frac{\delta\gamma}{d_L^2d_M^2}$, as required.
\end{proof}

The main goal of the rest of this section will be to prove the following theorem:

\begin{theorem} \label{thm:tight_additive}
	For any $\epsilon > 0$, there exists a $\delta$ such that any $\delta$-tight lattice is $\epsilon$-additive.
\end{theorem}

This will allow us to apply the numerical zig-zag lemma to a short exact sequence of chain complexes consisting of tight lattices.

\begin{proposition} \label{thm:tight_basis_1}
	If a locally $d$-dimensional lattice is generated by a set of vectors $S$ with $||x|| \leq \delta$ for all $x \in S$, then it has a basis consisting of vectors of norm $\leq d\delta$.
\end{proposition}

\begin{proof}
	Clearly this holds for the zero lattice. Let $d \in \{1, 2, 3 \ldots\}$, let $L$ be any lattice that is locally $d$-dimensional, and suppose the result holds for all lattices of local dimension less than $d$. If we can show that the result holds for $L$, then by induction it must hold for all lattices.

	Pick some $x \in S$. Let M be the normal sublattice $L \cap \spn(x)$. Let $f$ be the projection mapping from $L$ to $L / M$. We clearly have that $f(S)$ generates $L/M$, so by the induction hypothesis there exists a basis $B$ of $L/M$ consisting of vectors with norm at most $(d-1)\delta$. For each $b \in B$, we can find a vector $c \in L$ with $f(c) = b$ and $||c|| \leq d\delta$. To see this, take $c_0$ to be any element of $L$ in $f^{-1}(\{b\})$. Then if we write $c_0$ as $\lambda x + b$, we have that $c := c_0 - \left \lfloor{\lambda}\right \rfloor x$ satisfies the required properties. By choosing such a $c$ for each $b \in B$, we can construct a subset $C$ of $L$ that maps bijectively onto $B$ via $f$, and satisfies $||c|| \leq d\delta$ for all $c \in C$. Let $x'$ be a generator of $M$. We clearly have that $||x'|| \leq ||x|| \leq \delta \leq d\delta$, and it's straightforward to verify that $C \cup \{x'\}$ is a basis for $L$. So the result holds for L. Therefore, by induction, it holds for all lattices.
\end{proof}

\begin{corollary} \label{thm:tight_basis_2}
	Any $\delta$-tight lattice has a basis consisting of vectors with norm at most $\frac{\delta}{d}$
\end{corollary}

\begin{proposition} \label{thm:tight-1d}
	Suppose $\sigma:\mathbb{R} \rightarrow \mathbb{R}$ is given by $\sigma(x) = e^{-\pi(ax + b)^2}$ for some positive $a$ and real $b$. Then $|\sum\limits_{n \in \mathbb{Z}}\sigma(n) - \frac{1}{a}| \leq 2$.
\end{proposition}

\begin{proof}
	$$\sum\limits_{n \in \mathbb{Z}}\sigma(n) - \frac{1}{a} = \sum\limits_{n \in \mathbb{Z}}\sigma(n) - \int_{-\infty}^{\infty} \sigma(x) dx$$
	$$ = \int_{-\infty}^{\infty} \sigma'(x)(x - \left \lfloor{x}\right \rfloor ) dx$$
	$$|\sum\limits_{n \in \mathbb{Z}}\sigma(n) - \frac{1}{a}| = |\int_{-\infty}^{\infty} \sigma'(x)(x - \left \lfloor{x}\right \rfloor ) dx |$$
	$$\leq \int_{-\infty}^{\infty} |\sigma'(x)(x - \left \lfloor{x}\right \rfloor )| dx$$
	$$= \int_{-\infty}^{\infty} |\sigma'(x)| \cdot |(x - \left \lfloor{x}\right \rfloor )| dx$$
	$$\leq \int_{-\infty}^{\infty} |\sigma'(x)| dx$$
	$$= 2$$
\end{proof}

\begin{proposition} \label{thm:coset_mass_1}
	Suppose L is a locally d-dimensional lattice with a basis $b_1 \ldots b_d$ satisfying $\|b_i\| \leq \delta$ for $i \in \{1 \ldots d\}$. Then for any $x \in \spn(L)$, we have: $$(1 - 2\delta)^d \leq \vol(L)\rho(L+x) \leq (1 + 2\delta)^d$$
\end{proposition}

\begin{proof}
	When d = 1, this follows from \ref{thm:tight-1d}. For higher d, we prove the result by induction, supposing the result holds for all lattices of local dimension $< d$. Let $M$ be the normal sublattice of $L$ generated by $b_1 \ldots b_{d-1}$, and let $V = \spn(M)$. We can write $x$ as $x_1 + x_2$, where $x_1 \in V$ and $x_2 \in V^\perp$. Similarly, we can write $b_d$ as $y_1 + y_2$, where $y_1 \in V$ and $y_2 \in V^\perp$. Clearly $L/M$ is the lattice generated by $y_2$.
	We can break up $L + x$ as
	$$L + x = \bigcup \limits_{n \in \mathbb{Z}} (M + x + n\cdot b_d)$$
	from which it follows (using the induction hypothesis on the fourth and sixth lines):
	$$\rho(L + x) = \sum \limits_{n \in \mathbb{Z}}\rho(M + x + n\cdot b_d)$$
	$$= \sum \limits_{n \in \mathbb{Z}}\rho(M + x_1 + x_2 + n\cdot y_1 + n \cdot y_2)$$
	$$= \sum \limits_{n \in \mathbb{Z}}\rho(M + x_1 + n\cdot y_1)\rho(x_2 + n \cdot y_2)$$
	$$\leq \sum \limits_{n \in \mathbb{Z}}\frac{(1 + 2\delta)^{d - 1}}{\vol(M)}\rho(x_2 + n \cdot y_2) $$
	$$= \frac{(1 + 2\delta)^{d - 1}}{\vol(M)}\rho(L/M + x_2)$$
	$$\leq \frac{(1 + 2\delta)^{d - 1}}{\vol(M)}\cdot\frac{(1 + 2\delta)}{\vol(L/M)}$$
	$$= \frac{(1 + 2\delta)^d}{\vol(L)}$$
	$$\vol(L)\rho(L+x) \leq (1 + 2\delta)^d$$

	By a similar argument, we have $\vol(L)\rho(L+x) \geq (1 - 2\delta)^d$. By induction, we have that these inequalities hold for all d.
\end{proof}

\begin{corollary} \label{thm:coset_mass_2}
	$$2-e^{2d\delta} \leq \vol(L)\rho(L+x) \leq e^{2d\delta}$$
\end{corollary}

\begin{proof}
	We easily get $2 - (1 + 2\delta)^d \leq (1 - 2\delta)^d$ from expanding both sides. We also get $(1 + 2\delta)^d \leq e^{2d\delta}$ from $1 + 2\delta \leq e^{2\delta}$. Combining these inequalities with \ref{thm:coset_mass_1}, we have
	$$2-e^{2d\delta} \leq 2 - (1 + 2\delta)^d \leq (1 - 2\delta)^d \leq \vol(L)\rho(L+x)$$
	$$\leq (1 + 2\delta)^d \leq e^{2d\delta}$$
\end{proof}

\begin{corollary} \label{thm:coset_mass_3}
	$$|\vol(L)\rho(L+x) - 1| \leq e^{2d\delta} - 1$$
\end{corollary}

\begin{proof}
	This follows immediately from \ref{thm:coset_mass_2}.
\end{proof}

\begin{theorem} \label{thm:coset_mass_4}
	For all $\epsilon >0$ there exists a $\delta > 0$ such that for any $\delta$-tight lattice $L$, and any $x \in \spn(L)$, we have $|\ln(\vol(L)\rho(L+x))| < \epsilon$.
\end{theorem}

\begin{proof}
	Since $\ln$ is continuous and $\ln(1) = 0$, we can find a $\gamma > 0$ such that $|\ln(a)| < \epsilon$ whenever $|a - 1| \leq \gamma$. Take $\delta = \frac{\ln(\gamma + 1)}{2}$. Then $\gamma = e^{2\delta} - 1$. Now whenever a lattice L is $\delta$-tight, we know from \ref{thm:tight_basis_2} that it has a basis of vectors with norm at most $\frac{\delta}{d}$, and so from \ref{thm:coset_mass_3} we have
	$$|\vol(L)\rho(L+x) - 1| \leq e^{2\delta} - 1 = \gamma$$
	and therefore
	$$|\ln(\vol(L)\cdot\rho(L+x))| < \epsilon$$
\end{proof}

\begin{proof} [Proof of Theorem \ref{thm:tight_additive}]
	Use \ref{thm:coset_mass_4} to pick a $\delta$ so that, for any $\delta$-tight lattice $L$ and any $x \in \spn(L)$, 
	$|\ln(\vol(L)\cdot\rho(L+x))| < \epsilon / 2$.
	
	Then, for any $x \in \spn(L)$, we have both $|\ln(\vol(L)) + \ln(\rho(L))| < \epsilon / 2$ and $|\ln(\vol(L)) + \ln(\rho(L+x))| < \epsilon / 2$, so from the triangle inequality we have $|\ln(\rho(L)) - \ln(\rho(L+x))| < \epsilon$. Therefore, the following holds whenever $L$ is $\delta$-tight:
	
	\begin{equation} \label{eqn:rho_inf}
		\inf_{x \in \spn(L)}\ln(\rho(L + x)) \geq \ln(\rho(L)) - \epsilon = \ell(L) - \epsilon
	\end{equation}

	Now suppose $L$ is $\delta$-tight and $0 \rightarrow L \xrightarrow{f} M \xrightarrow{g} N \rightarrow 0$ is a short exact sequence of lattices.
	
	As noted in section 7, we can assume without loss of generality that $L$ is a normal sublattice of $M$ and $N$ is the quotient lattice $L/M$. Then $g$ is the projection onto $\spn(L)^\perp$.
	
	We can break down the sum $\rho(M)$ into pieces corresponding to different cosets of $L$. Then the bound \eqref{eqn:rho_inf} gives us a bound on $\rho(M)$. More precisely,
	
	$$\rho(M)  = \sum_{y \in N}\sum_{\substack{z \in M \\ g(z) = y}}\rho(z)$$
	$$= \sum_{y \in N}\rho(y)\sum_{\substack{z \in M \\ g(z) = y}}\rho(z - y)$$
	$$\geq \sum_{y \in N}\rho(y)\inf_{x \in \spn(L)}\rho(L + x)$$
	$$= \rho(N)\inf_{x \in \spn(L)}\rho(L + x)$$
	
	Taking the logarithm of both sides,
	
	$$\ell(M) \geq \ell(N) + \ln(\inf_{x \in \spn(L)}\rho(L + x))$$
	$$= \ell(N) + \inf_{x \in \spn(L)}\ln(\rho(L + x))$$
	
	Finally, by applying \ref{eqn:rho_inf}, we have:
	$$\ell(M) \geq \ell(N) + \ell(L) - \epsilon$$
\end{proof}

\section{The Cohomology of Lattices}

In this section, we define a sequence of functions $h = (h^0, h^1, \ldots)$ on the objects of $\Lat$, and show that $h$ is an effaceable numerical $\delta$-functor. We thereby prove that cohomology exists on $\Lat$, in the sense of \ref{def:coh_exists2}. The numerical zig-zag lemma plays a key role.

First, we'll look at how to construct ``tight resolutions'' of any object.

\begin{definition} \label{def:canonical_resolutions}
	For $n \in \{1, 2, 3, \ldots\}$, we define $P_n$ to be the sublattice of $\mathbb{R}^2$ generated by the vectors $(\frac{1}{n}, 0)$ and $(0, \sqrt{\frac{1}{n^2} - \frac{1}{n^4}})$. $P_n$ is clearly a $\frac{4}{n}$-tight lattice. We also have a normal embedding $f_n: \mathbb{Z} \rightarrow P_n$ which sends $1 \in \mathbb{Z}$ to the vector $(\frac{1}{n}, \sqrt{1 - \frac{1}{n^2}})$. We therefore have a short exact sequence
	$$0 \rightarrow \mathbb{Z} \rightarrow P_n \rightarrow P_n / \mathbb{Z} \rightarrow 0.$$
	From \ref{thm:tensor_exact}, tensoring with an arbitrary lattice $L$ yields a short exact sequence
	$$0 \rightarrow L \rightarrow P_n \otimes L \rightarrow (P_n / \mathbb{Z}) \otimes L \rightarrow 0.$$ The latter terms of this short exact sequence yield the following chain complex:
	$$0 \rightarrow P_n \otimes L \rightarrow (P_n / \mathbb{Z}) \otimes L \rightarrow 0 \rightarrow 0 \rightarrow \ldots$$
	We'll call this sequence $R_{(L, n)}$. 
\end{definition}

\begin{remark}
	Our approach will be to define the cohomology of a lattice $L$ in terms of the cohomology of $R_{(L, n)}$ as $n \rightarrow \infty$; this parallels the general tendency in cohomology to define the cohomology of an object in terms of the cohomology of a resolution. Our approach resembles \v{C}ech cohomology particularly closely: we take resolutions of a certain object, which is the identity of the tensor product, and then tensor that resolution with other objects to get resolutions of arbitrary objects. This is essentially how the resolutions in \v{C}ech cohomology are made, although this isn't always stated explicitly.
\end{remark}

\begin{proposition} \label{thm:eventually_tight}
	For any lattice $L$ and any $\delta > 0$, we have for sufficiently large $n$ that $P_n \otimes L$ is $\delta$-tight.
\end{proposition}
\begin{proof}
	Clearly $L$ is $\gamma$-tight for some $\gamma$. Since $P_n$ is $\frac{4}{n}$-tight, $P_n$ is $\frac{\delta}{\gamma}$-tight for $n \geq \frac{4\gamma}{\delta}$. Then $P_n \otimes L$ is $\delta$-tight from \ref{thm:tensor_tight}.
\end{proof}

\begin{corollary} \label{thm:tight_resolution}
For any lattice $L$, and any $\delta > 0$, there is short exact sequence $0 \rightarrow L \rightarrow M \rightarrow N \rightarrow 0$ with $M$ and $N$ both $\delta$-tight.
\end{corollary}
\begin{proof}
	We know for sufficiently large $n$ that $P_n \otimes L$ is $\delta$-tight. Furthermore, since quotients of $\delta$-tight lattices are also clearly $\delta$-tight, we must have that $(P_n / \mathbb{Z}) \otimes L \simeq (P_n \otimes L) / L$ is $\delta$-tight in this case. Since we have the short exact sequence
	$$0 \rightarrow L \rightarrow P_n \otimes L \rightarrow (P_n / \mathbb{Z}) \otimes L \rightarrow 0,$$
	we can take $M = P_n \otimes L$ and $N = (P_n / \mathbb{Z}) \otimes L$.
\end{proof}

\begin{proposition} \label{thm:additive_acyclic_1}
	If a lattice $L$ is $\epsilon$-additive, then for all $n$ we have $h_\bullet^1(R_{L, n}) \leq \epsilon$. 
\end{proposition}
\begin{proof}
	We have a short exact sequence $0 \rightarrow L \rightarrow P_n \otimes L \rightarrow (P_n / \mathbb{Z}) \otimes L \rightarrow 0$. Since $L$ is $\epsilon$-additive, we know
	$$\ell((P_n / \mathbb{Z}) \otimes L) - \ell(P_n \otimes L) + \ell(L) \leq \epsilon.$$
	We also know from left-exactness that the quantity on the left is equal to $h_\bullet^1(R_{L, n})$.
\end{proof}
\begin{definition} \label{def:epsilon_exact}
	A sequence $e_0, e_1, e_2, \ldots$ of real numbers is \textbf{$\epsilon$-exact} if all $e_k$ are nonnegative and we have $\sum\limits_{j=0}^{k} (-1)^j e_{k-j} \geq -\epsilon$ for all $k \geq 0$ (compare with \ref{def:numex}).
\end{definition}
\begin{proposition}  \label{thm:eventually_exact}
	Suppose $0 \rightarrow L \rightarrow M \rightarrow N \rightarrow 0$ is a short exact sequence of lattices. For any $\epsilon > 0$, the following sequence is $\epsilon$-exact for sufficiently large $n$:
	$$h_\bullet^0(R_{L, n}), h_\bullet^0(R_{M, n}), h_\bullet^0(R_{N, n}), h_\bullet^1(R_{L, n}), h_\bullet^1(R_{M, n}), h_\bullet^1(R_{N, n}), 0, 0, \ldots$$
\end{proposition}
\begin{proof}
	Let $\gamma = \frac{\epsilon}{7}$. From \ref{thm:tight_additive}, we can pick a $\delta$ so that any $\delta$-tight lattice is $\gamma$-additive.	
	
	For any $n$, we can obtain the following commutative diagram by tensoring the chain complex $0 \rightarrow P_n \rightarrow P_n / \mathbb{Z} \rightarrow 0 \rightarrow 0 \ldots$ with the short exact sequence $0 \rightarrow L \rightarrow M \rightarrow N \rightarrow 0$:
	
	\begin{tikzpicture}
	\matrix (m) [
	matrix of math nodes,
	row sep=2.0em,
	column sep=2.0em,
	text height=1.5ex, text depth=0.25ex
	]
	{ 0 & P_n \otimes N & (P_n / \mathbb{Z}) \otimes N & 0 & 0 & \ldots\\
		0 & P_n \otimes M & (P_n / \mathbb{Z}) \otimes M & 0 & 0 & \ldots\\
		0 & P_n \otimes L & (P_n / \mathbb{Z}) \otimes L & 0 & 0 & \ldots\\
	};
	
	\path[overlay,->, font=\scriptsize,>=latex]
	(m-1-1) edge (m-1-2)
	(m-1-2) edge (m-1-3)
	(m-1-3) edge (m-1-4)
	(m-1-4) edge (m-1-5)
	(m-1-5) edge (m-1-6)
	(m-2-1) edge (m-2-2)
	(m-2-2) edge (m-2-3)
	(m-2-3) edge (m-2-4)
	(m-2-4) edge (m-2-5)
	(m-2-5) edge (m-2-6)
	(m-3-1) edge (m-3-2)
	(m-3-2) edge (m-3-3)
	(m-3-3) edge (m-3-4)
	(m-3-4) edge (m-3-5)
	(m-3-5) edge (m-3-6)
	(m-2-1) edge (m-1-1)
	(m-2-2) edge (m-1-2)
	(m-2-3) edge (m-1-3)
	(m-2-4) edge (m-1-4)
	(m-2-5) edge (m-1-5)
	(m-3-1) edge (m-2-1)
	(m-3-2) edge (m-2-2)
	(m-3-3) edge (m-2-3)
	(m-3-4) edge (m-2-4)
	(m-3-5) edge (m-2-5);
	\end{tikzpicture}
	
	We have from \ref{thm:tensor_exact} that the columns of this diagram are short exact sequences. From \ref{thm:eventually_tight}, and the fact that quotients of $\delta$-tight lattices are $\delta$-tight, we must have that the lattices on the bottom row are $\delta$-tight for sufficiently large $n$. In this case, the diagram above must be a $\gamma$-additive sequence of chain complexes from our choice of $\delta$. The result then follows from \ref{thm:zigzag}, the numerical zig-zag lemma.
\end{proof}

\begin{proposition} \label{thm:h1_converges}
	For any lattice $L$, the sequence $h_\bullet^1(R_{L, 1}), h_\bullet^1(R_{L, 2}), h_\bullet^1(R_{L, 3}), \ldots$ converges.
\end{proposition}
\begin{proof}
	Our strategy will be to show the sequence converges by showing that it's Cauchy.
	
	Fix an arbitrary $\epsilon > 0$. From \ref{thm:tight_additive} and \ref{thm:tight_resolution}, we can find a short exact sequence $0 \rightarrow L \rightarrow M \rightarrow N \rightarrow 0$ in which $M$ and $N$ are both $\epsilon$-additive.
	
	From \ref{thm:eventually_exact}, we have the following for sufficiently large $n$:
	$$|h_\bullet^0(R_{L, n}) - h_\bullet^0(R_{M, n}) + h_\bullet^0(R_{N, n}) - h_\bullet^1(R_{L, n}) + h_\bullet^1(R_{M, n}) - h_\bullet^1(R_{N, n})| \leq \epsilon$$
	From \ref{thm:additive_acyclic_1}, we know both $h_\bullet^1(R_{M, n})$ and $h_\bullet^1(R_{(N, n})$ are in $[0, \epsilon]$. Combining this knowledge with the previous equation, we have
	$$|h_\bullet^0(R_{L, n}) - h_\bullet^0(R_{M, n}) + h_\bullet^0(R_{N, n}) - h_\bullet^1(R_{L, n})| \leq 2\epsilon,$$
	or equivalently
	$$h_\bullet^1(R_{L, n}) \in [\ell(L) - \ell(M) + \ell(N) - 2\epsilon, \ell(L) - \ell(M) + \ell(N)  + 2\epsilon].$$
	Since this last statement holds for all sufficiently large $n$, we have for all sufficiently large $n$ and $n'$ that $|h_\bullet^1(R_{L, n}) - h_\bullet^1(R_{L, n'})| \leq 4\epsilon$. Since this holds for any $\epsilon > 0$, the sequence $h_\bullet^1(R_{L, 1}), h_\bullet^1(R_{L, 2}), h_\bullet^1(R_{L, 3}), \ldots$ is Cauchy. Therefore, it converges.
\end{proof}

\begin{definition} \label{def:lattice_cohomology}
	For any lattice $L$ and any $k \in \{0, 1, 2, \ldots\}$, we define $h^k(L)$ to be $\lim_{n\to\infty}h_\bullet^k(R_{L, n})$. This clearly converges for $k = 0$ since the sequence is constantly $\ell(L)$; it also converges for $k > 1$ as the sequence is constantly $0$, and it converges for $k = 1$ from \ref{thm:h1_converges}.
\end{definition}

\begin{proposition} \label{thm:is_delta_functor}
	The sequence of functions $h := (h^0, h^1, h^2, \ldots)$ on the objects of $\Lat$ is a numerical $\delta$-functor.
\end{proposition}
\begin{proof}
	This follows immediately from \ref{thm:eventually_exact}.
\end{proof}

\begin{proposition} \label{thm:additive_acyclic_2}
	If a lattice $L$ is $\epsilon$-additive, then $h^1(L) \leq \epsilon$
\end{proposition}
\begin{proof}
	This follows immediately from \ref{thm:additive_acyclic_1}.
\end{proof}
\begin{corollary} \label{thm:tight_acyclic}
	For any $\epsilon > 0$, there exists a $\delta > 0$ such that, for any $\delta$-tight lattice $L$, we have $h^1(L) \leq \epsilon$.
\end{corollary}
\begin{proof}
	This follows immediately from \ref{thm:additive_acyclic_2} and \ref{thm:tight_additive}.
\end{proof}

\begin{proposition} \label{thm:is_effaceable}
	The numerical $\delta$-functor $h = (h^0, h^1, h^2, \ldots)$ is effaceable.
\end{proposition}

\begin{proof}
	The functions $h^k$ for $k > 1$ are clearly effaceable since they are constantly $0$. Therefore, it suffices to show that, for any lattice $L$ and any $\epsilon > 0$, we can find a short exact sequence $0 \rightarrow L \rightarrow N \rightarrow M \rightarrow 0$ with $h^1(N) < \epsilon$. This is an immediate consequence of \ref{thm:tight_acyclic} and \ref{thm:tight_resolution}.
\end{proof}

\begin{theorem} \label{thm:coh_exists2}
	Cohomology exists on $\Lat$.
\end{theorem}
\begin{proof}
	This follows immediately from \ref{thm:is_delta_functor} and \ref{thm:is_effaceable}.
\end{proof}

\begin{remark}
	We could also prove \ref{thm:coh_exists2} by defining $h$ by the explicit formulas given in \ref{thm:explicit_cohomology}, rather than defining it via resolutions and then deriving the explicit formulas after. It can then be shown that this is a numerical $\delta$-functor by proving each of the required inequalities separately (there are fairly straightforward elementary proofs of each one). Effaceability can still be shown using \ref{thm:coset_mass_4} and the same short exact sequence used in \ref{thm:is_effaceable}. This approach wouldn't require the numerical zig-zag lemma, and would probably be simpler overall. However, the approach given here is less ad-hoc, more closely resembles other approaches to cohomology, and shows greater promise of generalizing to other contexts. In fact, as mentioned previously, it is already known to the author how the numerical cohomology of coherent sheaves on a projective scheme over a field can be defined from scratch (as opposed to taking the usual cohomology for granted, as in section 5), in way that parallels the approach given here. This works for projective schemes of any dimension. This will possibly be the topic of a future paper. The higher-dimensional arithmetic case is still open.
\end{remark}

\section{Explicit Description of Cohomology}

In this section, we work out what $h^k(L)$ is for a lattice $L$. We trivially have, for any lattice $L$, that $h^0(L) = \ell(L) = \ln(\rho(L))$ and $h^k(L) = 0$ for $k > 1$. The remaining case is $k = 1$. We'll find that our definition agrees with the one given in \cite{VDGS}.

\begin{definition} \label{def:chi_lattice}
	Since we have the numerical $\delta$-functor $h$ on $\Lat$, we can define for any lattice $L$ the quantity $\chi(L) = \chi_h(L) = h^0(L) - h^1(L)$ (recall \ref{def:chi}).
\end{definition}

\begin{proposition} \label{thm:tight_volume_chi}
	For any $\epsilon > 0$, we can find a $\delta$ so that any $\delta$-tight lattice $L$ satisfies $|\chi(L) + \ln(\vol(L))| \leq \epsilon$
\end{proposition}

\begin{proof}
	The idea is that, for small $\delta$, both $\chi(L)$ and $-\ln(\vol(L))$ are approximated by $h^0(L) = \ell(L)$.
	More precisely, we have from \ref{thm:tight_acyclic} that $h^1(L) \leq \epsilon/2$ for sufficiently small $\delta$. Also, from \ref{thm:coset_mass_4} we have that $|h^0(L) + \ln(\vol(L))| \leq \epsilon/2$ for sufficiently small $\delta$. Choose a $\delta$ small enough that both these properties hold. We then have:
	$$|\chi(L) + \ln(\vol(L))| = |h^0(L) - h^1(L) + \ln(\vol(L))|$$
	$$\leq |h^0(L) + \ln(\vol(L))| + h^1(L)$$
	$$\leq \frac{\epsilon}{2} + \frac{\epsilon}{2} = \epsilon$$
\end{proof}

\begin{proposition} \label{thm:volume_chi}
	For any lattice $L$, $\chi(L) = -\ln(\vol(L))$
\end{proposition}

\begin{proof}
	 Pick an $\epsilon > 0$. From \ref{thm:tight_volume_chi}, we can find a $\delta$ so that any $\delta$-tight lattice $M$ satisfies $|\chi(M) + \ln(\vol(M))| \leq \epsilon$. Furthermore, from \ref{thm:tight_resolution}, we have a short exact sequence $0 \rightarrow L \rightarrow M \rightarrow N \rightarrow 0$ with $M$ and $N$ both $\delta$-tight. We have both the following, from \ref{thm:chi_additive} and \ref{thm:exact_volume}:
	$$\chi(L) = \chi(M) - \chi(N)$$
	$$\ln(\vol(L)) = \ln(\vol(M)) - \ln(\vol(N))$$
	
	We can then add these equations, apply the triangle inequality and the bounds we get from our choice of $\delta$:
	
	$$|\chi(L) + \ln(\vol(L))|\leq |\chi(M) + \ln(\vol(M))| + |\chi(N) + \ln(\vol(N))|$$
	$$\leq \epsilon + \epsilon = 2\epsilon$$
	
	Since $|\chi(L) + \ln(\vol(L))| \leq 2\epsilon$ holds for any $\epsilon > 0$, we must have $\chi(L) = -\ln(\vol(L))$.
\end{proof}

\begin{corollary} \label{thm:explicit_h1}
	$h^1(L) = \ln(\rho(L)) + \ln(\vol(L))$
\end{corollary}

To summarize what we know about $h^k(L)$, we have:

\begin{theorem} \label{thm:explicit_cohomology}
	For any lattice $L$, we have that $h^0(L) = \ln(\rho(L))$, $h^1(L) = \ln(\rho(L)) + \ln(\vol(L))$, $h^k(L) = 0$ for $k \geq 2$, and $\chi(L) = -\ln(\vol(L))$.
\end{theorem}

\begin{remark}
	In \cite{VDGS}, $h^0$ of an Arakelov divisor $D$ is defined, with different notation, as $\ln(\rho(L))$, with $L$ the underlying lattice of $D$. They also propose the definition $h^1(D) = h^0(K - D)$, with $K$ the canonical divisor. It's pointed out that the corresponding lattice of $K - D$ is the dual lattice of $L$, and so it follows from the Poisson summation formula that $h^1(D) = \ln(\rho(L)) + \ln(\vol(L))$. They also define $\chi(D)$ as $-\ln(\vol(L))$. So, the definitions in \cite{VDGS} for the cohomology and Euler characteristic of an Arakelov divisor are consistent with the definitions proposed in this paper for the cohomology and Euler characteristic of a lattice.
\end{remark}

\begin{acknowledgements}
	Some of this work was completed at the Pacific Science Institute, so I'd like to thank Garrett Lisi and Crystal Baraniuk for creating a very pleasant working environment there.
	
	The latter part of this paper depends on a result recently proven by Oded Regev and Noah Stephens-Davidowitz. If they hadn't proven it, I'd probably still be stuck, so I'd like to thank them for their contribution.	
	
	I would also like to thank James Borger for helpful comments and proofreading.	
\end{acknowledgements}

\end{document}